\documentclass[a4paper,12pt]{amsart}
\usepackage{amsfonts}
\usepackage{amsmath,array}
\usepackage{amssymb}
\usepackage{mathrsfs}
\usepackage[colorlinks]{hyperref}
 \usepackage{graphicx}
\usepackage{enumerate}
\usepackage[usenames]{color}
\usepackage{cases}
\newtheorem{thm}{Theorem}[section]
\newtheorem{cor}[thm]{Corollary}
\newtheorem{lem}[thm]{Lemma}
\newtheorem{prop}[thm]{Proposition}

\numberwithin{equation}{section}
\usepackage[english]{babel} 
\usepackage{blindtext}

\theoremstyle{definition}
\newtheorem{definition}[thm]{Definition}
\newtheorem{rem}[thm]{Remark}

\begin{document}
 \title[H\"ormander type Fourier multiplier theorem and Nikolskii inequality]{H\"ormander type Fourier multiplier theorem and Nikolskii inequality on quantum tori, and applications}

\author[Michael Ruzhansky]{Michael Ruzhansky}
\address{
 Michael Ruzhansky:
  \endgraf
 Department of Mathematics: Analysis, Logic and Discrete Mathematics,
  \endgraf
 Ghent University, Ghent,
 \endgraf
  Belgium 
  \endgraf
  and 
 \endgraf
 School of Mathematical Sciences, Queen Mary University of London, London,
 \endgraf
 UK  
 \endgraf
  {\it E-mail address} {\rm michael.ruzhansky@ugent.be}
  }

\author[Serikbol Shaimardan]{Serikbol Shaimardan}
\address{
  Serikbol Shaimardan:
  \endgraf
  Institute of Mathematics and Mathematical Modeling, 050010, Almaty, 
  \endgraf
  Kazakhstan 
  \endgraf
  and 
  \endgraf
Department of Mathematics: Analysis, Logic and Discrete Mathematics
  \endgraf
 Ghent University, Ghent,
 \endgraf
  Belgium
  \endgraf
  {\it E-mail address} {\rm shaimardan.serik@gmail.com} 
  }

 \author[Kanat Tulenov]{Kanat Tulenov}
\address{
  Kanat Tulenov:
  \endgraf
  Institute of Mathematics and Mathematical Modeling, 050010, Almaty, 
  \endgraf
  Kazakhstan 
  \endgraf
  and 
  \endgraf
Department of Mathematics: Analysis, Logic and Discrete Mathematics
  \endgraf
 Ghent University, Ghent,
 \endgraf
  Belgium
  \endgraf
  {\it E-mail address} {\rm tulenov@math.kz} 
  }

\date{}

\begin{abstract}
In this paper, we study H\"ormander  
type Fourier multiplier theorem and the Nikolskii inequality on quantum tori. On the way to obtain these results, we also prove some classical inequalities such as the Paley, Hausdorff-Young-Paley, Hardy-Littlewood, and Logarithmic Sobolev inequalities on quantum tori. As applications we establish embedding theorems between Sobolev, Besov spaces as well as embeddings between Besov and Wiener and Beurling spaces on quantum tori. We also analyse $\beta$-versions of Wiener and Beurling spaces and their embeddings, and interpolation properties of all these spaces on quantum tori.  As an application of the study, we also derive a version of the Nash inequality, and the time decay for solutions of a heat type equation. 
\end{abstract}



\maketitle

\tableofcontents
{\section{Introduction}}
Quantum tori $\mathbb{T}^d_{\theta}$ is one of the main objectives in Alain Connes' non-commutative geometry  (see 
 \cite{Connes1980} and \cite{CM2014}, \cite[Chapter 12]{green-book}) and in non-commutative harmonic analysis \cite{LMcSZ, Ponge1, R, Rieffel, Spera, XXY}. This is also known as a non-commutative torus and irrational rotation algebra. In the last decades, there has been substantial body of work on generalising the methods of classical harmonic analysis on ordinary torus to quantum tori (or non-commutative torus). In particular, on a non-commutative torus (defined in terms of an arbitrary antisymmetric real $d\times d$ matrix $\theta$), it is possible to define analogues of many of the tools of harmonic analysis, such as differential operators and function spaces \cite{LMcSZ, MSQ2019, STZ, XXY}. 

The study of Fourier multipliers has been attracting attention of many researchers for more than a century. This is related to numerous applications in harmonic analysis, in particular, in PDEs. One of the important questions in this field is to understand the $L^p\to L^q$ boundedness of Fourier multipliers. In the case $p = q,$ remarkable results were obtained by Marcinkiewicz \cite{KM}, \cite{M} and Mikhlin \cite{Mikh}, which were improved later by Stein \cite{Stein} and H\"ormander \cite{Hor}. 

For the case $L^p\to L^q,$ there is another classical result of H\"ormander available in \cite[Theorem 1.11]{Hor}.
If $\mathcal{S}(\mathbb{R}^d)$ is the Schwartz class and $T_g$ is the Fourier
multiplier with the symbol $g\in \mathcal{S}(\mathbb{R}^d),$ which is given by the multiplication on the Fourier transform side
$$\widehat{T_g}(f)=g \cdot \widehat{f}, \quad f\in\mathcal{S}(\mathbb{R}^d),$$
where $\widehat{f}$ is the Fourier transform
$$\widehat{f}(t)=(2\pi)^{-d/2}\int_{\mathbb{R}^d}f(s)e^{-i(t,s)}ds, \quad t\in \mathbb{R}^d,$$ then the classical H\"ormander's result asserts as follows:
\begin{thm}\cite[Theorem 1.11]{Hor}\label{Hor-thm}
    If $1 < p\leq  2\leq q < \infty,$ then the Fourier multiplier $T_g$ with the symbol $g:\mathbb{R}^{d}\to\mathbb{C}$ such that
    $$\sup\limits_{\lambda>0}\lambda\left(\int\limits_{x\in\mathbb{R}^d \atop  |g(x)|\geq\lambda}dx\right)^{\frac{1}{p}-\frac{1}{q}}<\infty$$
     extends to a bounded map from $L^p(\mathbb{R}^{d})$ to $L^q(\mathbb{R}^{d})$ and we have
\begin{equation}\label{main est}\|T_g\|_{L^p(\mathbb{R}^{d})\to L^q(\mathbb{R}^{d})}\lesssim \sup\limits_{\lambda>0}\lambda\left(\int\limits_{x\in\mathbb{R}^d \atop  |g(x)|\geq\lambda}dx\right)^{\frac{1}{p}-\frac{1}{q}}. \end{equation}
\end{thm}
It can be found from \cite[p. 303]{E}
that under the same conditions on $p, q,$ and $r,$ for a sequence of complex numbers $g=\{g(k)\}_{k\in\mathbb{Z}^{d}}$ and the corresponding Fourier series multiplier $T_g,$ the same result as in Theorem \ref{Hor-thm} holds. There are also other works in this direction, see e.g. \cite{AMR, ARN, CK, K2, K3, K4, M-Nur, NT1, NT2, RV, RT, RST1, RST2, Z}. For example, in \cite{AR}, authors obtained $L^p- L^q$ H\"ormander type theorem in the context of locally compact groups with applications such as the $L^p- L^q$ boundedness of the heat semigroup and Sobolev embedding theorem in which our results are motivated. Fourier multipliers on quantum tori and their $L^p \to L^p$ completely boundedness properties were studied in \cite{CXY, R, XXY} (see also \cite{STZ} in more general settings of quasi-Banach symmetric spaces). 

In this work, we study the $L^p\to L^q$ boundedness of the Fourier multipliers on quantum tori. In other words, we obtain a quantum analogue of the above Theorem \ref{Hor-thm} on quantum tori (see Theorem \ref{thm5.1}). We should also mention that after we submitted this paper, the same result was obtained very recently in \cite{STT} by a simpler approach as in \cite{Z}. On the way to achieve our goal, in this work, we prove a number of classical inequalities such as Paley, Hausdorff-Young-Paley, Hardy-Littlewood, and Logarithmic Sobolev inequalities on quantum tori which also have further applications in PDEs. As a simple application of the H\"ormander type theorem we obtain a Sobolev type inequality (see Theorem \ref{sobolev-embed}) on quantum tori. As an application of our Nash-type inequality we compute the decay rate for the heat equation for the sub-Laplacian in our setting as in the classical case.

As a second main topic we analyze the families of Besov, Wiener and Beurling spaces on quantum tori. To a large extent, the analysis is based on establishing an appropriate version of Nikolskii’s inequality, which is also known as the reverse H\"{o}lder inequality, in this non-commutative setting. The classical Nikolskii inequality for trigonometric polynomials on the circle was given by \cite{Nik} (see also \cite[Chapter III.3.2, p.  147]{HH1987}). This inequality plays an important role in investigating basic properties of different function spaces \cite{Nik-book, HH1987}, in approximation theory \cite{DL}, and in obtaining embedding theorems \cite{DT, HH1987}. The Nikolskii inequality for the spherical polynomials was proved in \cite[Theorem 1]{K}, see also \cite{Dai}. Furthermore, in \cite{P1} and \cite{P2} Pesenson obtained a very general version of Bernstein-Nikolskii type inequalities on compact homogeneous  and non-compact symmetric spaces, respectively. In this work, we prove the Nikolskii inequality for trigonometric polynomials in non-commutative $L^p$-spaces associated with the quantum tori. In the special case where $\theta= 0,$ our result Theorem \ref{thm3.2} coincides with \cite[Proposition, p. 147]{HH1987} and extends Pesenson’s result \cite{P1} to a wider range of indices $0 < p\leq q\leq \infty$ with a completely different proof approach. Our method is based on the idea in \cite{NRT}, where the authors obtained a similar inequality on compact homogeneous manifolds.
As applications of the Nikolskii inequality, we establish embedding theorems between Besov spaces, as well as embeddings between Besov and Wiener and Beurling spaces on quantum tori. Generally, Besov spaces on quantum tori and their embedding properties were studied in \cite[Chapter V and VI]{XXY} in the case when $1\leq p,q\leq \infty.$ We study similar questions extending the range of indices to $0<p,q\leq \infty$ by using Nikolskii inequality. Apart from embeddings between Besov spaces that can be obtained by trivial arguments, the Nikolskii inequality allows us to derive a rather complete list of embeddings with respect to all three indices of the space which we address in Theorem \ref{embed_Besov}.  Finally, we analyze $\beta$-versions of Wiener and Beurling spaces and their embeddings, and the interpolation properties of all these spaces in quantum tori in Section \ref{inter}.


\section{Preliminaries}
 
 \subsection{Quantum tori} Let $d\geq 2$ be an integer and let $\theta$ be a real anti-symmetric $d\times d$-matrix. We denote by $C(\mathbb{T}^{d}_{\theta})$ a $\ast$-algebra generated by elements $\left(U_k\right)_{1\leq{k}\leq{d}}$ satisfying the condition 
\begin{eqnarray*}
U_kU_m=e^{i\theta_{km}}U_mU_k, \;\;\;U_kU^*_k=U^*_kU_k=1.
\end{eqnarray*}
For any $n\in\mathbb{Z}^d,$ define $e^\theta_n$ by
\begin{equation}\label{basis}
   e^\theta_n=U^{n_1}_1U^{n_2}_2\dots U^{n_d}_d,\quad n=(n_1,\dots,n_d)\in\mathbb{Z}^d. 
\end{equation}
Note that 
\begin{eqnarray}\label{additive2.1}
e^\theta_ne^\theta_m=e^{-i(\sum\limits_{j<k}\theta_{jk}n_jm_k)}e^\theta_{n+m},\;\;\;(e^\theta_n)^*=e^{-i(\sum\limits_{j<k}\theta_{jk}n_jn_k)}e^\theta_{-n}.  
\end{eqnarray}
A polynomial in $C(\mathbb{T}^{d}_{\theta})$ is a finite sum 
\begin{eqnarray*}
x= \sum\limits_{m\in\mathbb{Z}^d}\alpha(m) e^\theta_m,\quad\alpha(m)\in\mathbb{C},\quad m\in\mathbb{Z}^d,    
\end{eqnarray*}
that is, $\alpha(m) = 0$ for all but finite indices $m\in\mathbb{Z}^d.$
The involution algebra of all such polynomials
is dense in $C(\mathbb{T}^{d}_{\theta})$ (see, \cite{Ponge1}). For any polynomial $x$ as above, we define a linear functional $\tau_\theta$ by the formula 
\begin{equation}\label{additive2.2}\tau_{\theta}(x) = \alpha(0),\,\ 0 = (0, \ldots, 0) \in \mathbb{Z}^{d}.
\end{equation}
This \(\tau_{\theta}\) extends to a normal, faithful, finite tracial state on $C(\mathbb{T}^{d}_{\theta})$ (see, \cite{Ponge1}).
By formulas (\ref{additive2.1})-(\ref{additive2.2}) and by a simple computation we find some relations which will be useful for our further investigations. Indeed, we have that 
\begin{equation}\label{additive2.3}
 |\tau_\theta(e^\theta_n e^\theta_m)|\overset{ \text{(\ref{additive2.1}})}{=}|e^{-i(\sum\limits_{j<k}\theta_{jk}n_jm_k)}\tau_\theta(e^\theta_{n+m})|\overset{ \text{(\ref{additive2.2}})}{=}|e^{i(\sum\limits_{j<k}\theta_{jk}n_jm_k)}||\tau_\theta(e^\theta_{n+m})|\overset{ \text{(\ref{additive2.2}})}{=}1 
\end{equation}
whenever $n+m=0,$
and 
\begin{eqnarray}\label{additive2.4}
 \tau_\theta((e^\theta_n)^*e^\theta_n)\overset{ \text{(\ref{additive2.1}})}{=}\tau_\theta(e^{-i(\sum\limits_{j<k}\theta_{jk}n_jn_k)}e^\theta_{-n}e^\theta_n) \overset{ \text{(\ref{additive2.1}})}{=}\tau_\theta(e^\theta_0)\overset{ \text{(\ref{additive2.2}})}{=}1.
\end{eqnarray}
Also, $ \tau_\theta((e^\theta_n)^*e^\theta_m)=0$ for $n\neq m.$
Moreover, we have
\begin{eqnarray}\label{Plancerel-relation}
\nonumber\tau_\theta((\sum\limits_{m\in\mathbb{Z}^d}\alpha(m) e^\theta_m)(\sum\limits_{n\in\mathbb{Z}^d}\beta(n) e^\theta_n))&=&\sum\limits_{m,n\in\mathbb{Z}^d}\alpha(m)\beta(n) \tau_\theta\left(e^\theta_me^\theta_n\right)\\
&\overset{ \text{(\ref{additive2.3}})}{=}&\sum\limits_{n\in\mathbb{Z}^d}e^{i(\sum\limits_{j<k}\theta_{jk}n_jn_k)}\alpha(-n)  \beta(n),
\end{eqnarray}
which implies that 
\begin{eqnarray*}
\tau_\theta\left(xy\right)=\tau_\theta\left(yx\right),\;\;\;x,y\in C(\mathbb{T}^{d}_{\theta}).
\end{eqnarray*}
Let us equip $C(\mathbb{T}^{d}_{\theta})$ with an inner product defined by the formula 
\begin{eqnarray*}
(x, y)=\tau_\theta\left(xy^*\right).
\end{eqnarray*}
Then $(C(\mathbb{T}^{d}_{\theta}),(\cdot,\cdot))$ becomes a pre-Hilbert space. Let $L^2(\mathbb{T}^{d}_{\theta})$ be the completion of $(C(\mathbb{T}^{d}_{\theta}),(\cdot,\cdot))$ defined by the (nondegenerate) above inner product, which makes $H := L^2(\mathbb{T}^{d}_{\theta})$ a separable Hilbert space. Obviously, $\{e^{\theta}_n\}_{n\in\mathbb{Z}^d}$ is an orthonormal basis in $L^2(\mathbb{T}^{d}_{\theta}).$ We frequently refer to $\{e^{\theta}_n\}_{n\in\mathbb{Z}^d}$ as a trigonometric basis (with which it coincides for $\theta=0$).
We represent $C(\mathbb{T}^{d}_{\theta})$ on $L^2(\mathbb{T}^{d}_{\theta})$ by left multiplication $\lambda_l(x)y := xy,$ and denote by $L^{\infty}(\mathbb{T}^{d}_{\theta})$ the corresponding weak$^*$ closure of $\lambda_{l}(C(\mathbb{T}^{d}_{\theta}))$, a von Neumann algebra with operator norm denoted by $\|\cdot\|_{\infty}$.  Recall that the von Neumann algebra $L^{\infty}(\mathbb{T}^d_\theta)$  
is hyper finite \cite{XXY}.
For $x=\sum\limits_{n\in\mathbb{Z}^d}\alpha(n)e^\theta_n$, we have 
\begin{eqnarray}\label{additive2.6}
\tau_\theta\left(x^*x\right)&=&\tau_\theta\left(\left(\sum\limits_{m\in\mathbb{Z}^d}\alpha(m) e^\theta_m\right)^*\left(\sum\limits_{n\in\mathbb{Z}^d}\alpha(n) e^\theta_n\right)\right)\nonumber\\&=&\sum\limits_{m,n\in\mathbb{Z}^d}\overline{\alpha(m)}\alpha(n)\tau_\theta\left((e^\theta_m)^*e^\theta_n\right)\nonumber\\&=&\sum\limits_{n\in\mathbb{Z}^d}|\alpha(n)|^2.
\end{eqnarray}

Therefore, $\tau_\theta\left(x^*x\right)=0$ implies that $x=0$. 
 Both $\tau_{\theta}$ and $\lambda_l$ extend continuously to $L^{\infty}(\mathbb{T}^{d}_{\theta}),$ giving a faithful tracial state $\tau_{\theta}: L^{\infty}(\mathbb{T}^{d}_{\theta})\to\mathbb{C}$ and a faithful representation (injective $*$-homomorphism) $\lambda_l : L^{\infty}(\mathbb{T}^{d}_{\theta})\to \mathbb{B}(L^2(\mathbb{T}^{d}_{\theta})),$ where $\mathbb{B}(L^2(\mathbb{T}^{d}_{\theta}))$ is the algebra of all linear bounded operators on $L^2(\mathbb{T}^{d}_{\theta}).$
We identify $L^{\infty}(\mathbb{T}^{d}_{\theta})\subseteq L^{2}(\mathbb{T}^{d}_{\theta}).$
In fact, the unitaries $U_1,...,U_d$ can be concretely realized as unitary operators on $L^2(\mathbb{T}^{d}).$ Indeed, let $\theta$ be a real anti-symmetric $d\times d$-matrix and $\theta_1,...,\theta_d$ be its column vectors. Then, for $j=1,...,d$ we define the unitary operators 
$$(U_j\xi)(s)=e^{i s_j}\xi(s+\pi\theta_j), \quad \xi \in L^2(\mathbb{T}^d),$$ 
which satisfy the relations,
$$U_k U_j=e^{2\pi i \theta_{jk}}U_jU_k, \quad j,k=1,...,d.$$
In this case, our $C^*$-algebra $C(\mathbb{T}^{d}_{\theta})$ is generated by the above unitaries. For more details, we refer the reader to \cite{Ponge1}. 
If $\theta=0$, then our spaces $C(\mathbb{T}^{d}_{\theta})$ and $L^{\infty}(\mathbb{T}^{d}_{\theta})$ coincide with the spaces of continuous functions $C(\mathbb{T}^{d})$ and essentially bounded functions $L^{\infty}(\mathbb{T}^{d})$ on the ordinary $d$-torus $\mathbb{T}^d=\mathbb{R}^d/2\pi \mathbb{Z}^d.$
The pair $(L^{\infty}(\mathbb{T}^{d}_{\theta}),\tau_{\theta})$ is called a non-commutative measure space. As usual, we define $L^p(\mathbb{T}^d_{\theta})$ as a non-commutative $L^p$-space associated with the von Neumann algebra $L^{\infty}(\mathbb{T}^d_{\theta})$ and the trace $\tau_{\theta}.$ For more details on non-commutative $L^p$ spaces associated with a general semifinite von Neumann algebra we refer the reader to \cite{DdePS}, \cite{LSZ2012}, \cite{PXu}. Indeed, for any $0< p<\infty,$ we can define the $L^p$-norm on this space by the Borel functional calculus and the following formula:
$$\|x\|_{p}=\Big(\tau_{\theta}(|x|^p)\Big)^{1/p},\quad x\in L^{\infty}(\mathbb{T}^{d}_{\theta}), $$
where $|x|:=(x^{*}x)^{1/2}.$
The completion of $\{x\in L^{\infty}(\mathbb{T}^{d}_{\theta}) :\|x\|_{p}<\infty\}$ with respect to $\|\cdot\|_{p}$ is denoted by $L^p(\mathbb{T}^d_{\theta}).$ The elements of $L^p(\mathbb{T}^d_{\theta})$ are $\tau_{\theta}$-measurable operators like in the commutative case. These are linear densely defined closed (possibly unbounded) affiliated with $L^{\infty}(\mathbb{T}^{d}_{\theta})$ operators. For $x\in L^1(\mathbb{T}^d_{\theta})$ we can define the formal Fourier series of $x$ by
$$x\sim\sum_{n\in\mathbb{Z}^d}\widehat{x}(n)e^{\theta}_n
$$
with $\widehat{x}(n)=\tau_{\theta}(x(e^{\theta}_n)^{\ast}), \,\ n\in\mathbb{Z}^{d}.$
Every element $x\in L^2(\mathbb{T}^d_{\theta})$ admits a unique representation of the form
\begin{equation}\label{fourier-ser}x=\sum_{n\in\mathbb{Z}^d}\widehat{x}(n)e^{\theta}_n,\quad \{\widehat{x}(n)\}_{n\in\mathbb{Z}^d}\in \ell^2(\mathbb{Z}^d),
\end{equation}
where the series converges in $L^2(\mathbb{T}^d_{\theta})$-norm and
\begin{equation}\label{fourier-coef}
\widehat{x}(n)=\tau_{\theta}(x(e^{\theta}_n)^{\ast}), \,\ n\in\mathbb{Z}^{d}.
\end{equation}
Moreover, we have the Plancherel (or Parseval) identity (see, 
\cite[Formula 2.3]{MSQ2019})
\begin{eqnarray}\label{additive2.8}
\|x\|_{L^2\left(\mathbb{T}^d_\theta\right)}=\|\widehat{x}\|_{\ell^2\left(\mathbb{Z}^d\right)}. 
\end{eqnarray}
\begin{rem}\label{rem2.1} It is well known from the general theory of linear bounded operators that the   composition of two unitary operators and any integer power of unitary operators are also unitary operators. Since  $e^\theta_n$ is defined by a composition of unitary operators (see \eqref{basis}), it is itself a unitary operator for all $n\in\mathbb{Z}^d$ and we have
\begin{eqnarray}\label{additive2.5}
\|e^\theta_n\|_{ L^\infty(\mathbb{T}^d_\theta)}=\|(e^\theta_n)^*\|_{ L^\infty(\mathbb{T}^d_\theta)}=1. 
\end{eqnarray}
\end{rem}

The space $C^{\infty}(\mathbb{T}^d_\theta)$ is defined to be the set of $x=\sum\limits_{n\in\mathbb{Z}^d}\widehat{x}(n)e_n^{\theta}$ such that the sequence of Fourier coefficients $\{\widehat{x}(n)\}_{n\in\mathbb{Z}^d}$ has a rapid decay (the sequence $\{\widehat{x}(n)\}_{n\in\mathbb{Z}^d}$ is
eventually dominated by the reciprocal of any polynomial, that is for any $m\in \mathbb{Z}_+$ there is $c_m>0$ such that $|\widehat{x}(n)|\leq c_m (1+|n|)^{-m},  n\in\mathbb{Z}^d,$ where $|\cdot|$ is Euclidean length on $\mathbb{Z}^d$). This is a weak$^*$-dense subalgebra of $L^{\infty}(\mathbb{T}^{d}_{\theta})$ that contains all polynomials. We may consider $C^{\infty}(\mathbb{T}^d_\theta)$ as a non-commutative deformation of the space of smooth functions on $\mathbb{T}^d.$
There is also a canonical Fr\'{e}chet topology
on $C^{\infty}(\mathbb{T}^d_\theta).$ The space $\mathcal{D}'(\mathbb{T}^d_\theta)$ is defined to be the topological dual of $C^{\infty}(\mathbb{T}^d_\theta)$ and is called the space of distributions on $L^\infty(\mathbb{T}^d_\theta).$ In other words, the dual space $\mathcal{D}'(\mathbb{T}^d_\theta)$ is endowed with the weak$^*$-topology. We denote by $\langle\cdot,\cdot\rangle$
the duality bracket between $C^{\infty}(\mathbb{T}^d_\theta)$ and $\mathcal{D}'(\mathbb{T}^d_\theta)$, that is 
$F(x):=\langle F,x\rangle$ for $x\in C^{\infty}(\mathbb{T}^d_\theta)$ and $F\in \mathcal{D}'(\mathbb{T}^d_\theta).$ Moreover, for $1\leq p\leq \infty,$ the space $L^p(\mathbb{T}^d_{\theta})$ is naturally embedded into $\mathcal{D}'(\mathbb{T}^d_\theta).$ For more details, see \cite{XXY}. 

Throughout this paper, we write $\mathcal{A}\lesssim \mathcal{B}$ if there is a constant $c> 0$ depending only on parameters of spaces such that
$\mathcal{A}\leq c\mathcal{B}.$ We write $\mathcal{A}\asymp \mathcal{B}$ if both $\mathcal{A}\lesssim \mathcal{B}$ and $\mathcal{A}\gtrsim \mathcal{B}$ hold, possibly with different constants.

 \subsection{Fourier  multipliers on quantum tori} 
 Let $g:\mathbb{Z}^d\rightarrow\mathbb{C}$  be a sequence on $\mathbb{Z}^d$. The Fourier multiplier $T_g$ is defined on $x\in L^2( \mathbb{T}^d_\theta)$ with the symbol $g$ by
\begin{eqnarray}\label{additive2.9}
T_g(x)=\sum\limits_{n\in\mathbb{Z}^d}g(n)\widehat{x}(n)e^\theta_n,\;\;\;x \in L^2( \mathbb{T}^d_\theta).  
\end{eqnarray}
If $g$ is a bounded  sequence, then by virtue of the Plancherel identity (\ref{additive2.8}) and the uniqueness of the representation \eqref{fourier-ser}, $T_g$ indeed defines a bounded linear operator on $L^2( \mathbb{T}^d_\theta)$ and the above series converges in $L^2( \mathbb{T}^d_\theta)$-sense. If $g$ is unbounded, we may define $T_g$  on the dense subspace of $L^2(\mathbb{T}^d_\theta)$ of those $x$ with finitely many non-zero Fourier coefficients. 

\begin{definition}\label{def-adj}
Let $1 \leq p,q \leq \infty$ and $p',q'$ be the corresponding H\"{o}lder conjugates of $p,q.$  For a bounded linear operator $T:L^p(\mathbb{T}^d_\theta)\to L^q(\mathbb{T}^d_\theta)$ we denote by $T^* :L^{q'}(\mathbb{T}^d_\theta)\to L^{p'}(\mathbb{T}^d_\theta)$  its adjoint operator defined by
\begin{eqnarray}\label{def-dul-operator}
 (T(x),y):=\tau_{\theta}(T(x) y^*)=\tau_{\theta}(x(T^*(y))^*)=(x,T^*(y))   \end{eqnarray} 
for all $x\in L^p(\mathbb{T}^d_\theta)$ and $y\in L^{q'}(\mathbb{T}^d_\theta)$ (or in a dense subspace of it).  
\end{definition}

\begin{lem}\label{Duality}Let $1\leq p,q\leq \infty$ and let $T_g:L^p(\mathbb{T}^d_\theta)\to L^q(\mathbb{T}^d_\theta)$ be the Fourier multiplier defined by \eqref{additive2.9} with the symbol $g.$ Then its adjoint $T^*_{g}:L^{q'}(\mathbb{T}^d_\theta)\to L^{p'}(\mathbb{T}^d_\theta)$ equals to $T_{\overline{g}},$ where $\overline{g}$ is the complex conjugate of the function $g,$ in the sense of Definition \ref{def-adj}.    
\end{lem}
\begin{proof} By formulas \eqref{additive2.2} and \eqref{additive2.9} we write

\begin{eqnarray*}
( x,T^*_{g}(y))&\overset{ \text{\eqref{def-dul-operator}}}{=}&\tau_{\theta}(x(T_{g}^*(y))^*)=\tau_{\theta}(T_{g}(x) y^*)\\
 &\overset{\text{(\ref{additive2.9})}}{=}&\tau_{\theta}\Big( \sum\limits_{n\in\mathbb{Z}^d}g(n)\widehat{x}(n)e_n^\theta
 \big(\sum\limits_{m\in\mathbb{Z}^d} \widehat{y}(m)e_m^\theta\big)^*\Big)\\
 &\overset{ \text{(\ref{additive2.2})}}{=}&\sum\limits_{m,n\in\mathbb{Z}^d}g(n)\widehat{x}(n)\overline{\widehat{y}(m)}\tau_{\theta} \big(e_n^\theta 
   (e_m^\theta)^*\big)\\
&=&\tau_{\theta}\Big( \sum\limits_{n\in\mathbb{Z}^d} \widehat{x}(n)e_n^\theta  \sum\limits_{m\in\mathbb{Z}^d} g(m)\overline{\widehat{y}(m)}(e_m^\theta)^*\Big)\\
&=&\tau_{\theta}\Big( \sum\limits_{n\in\mathbb{Z}^d} \widehat{x}(n)e_n^\theta  \big(\sum\limits_{m\in\mathbb{Z}^d} \overline{g(m)}\widehat{y}(m)e_m^\theta\big)^*\Big)\\
&=&(x,T_{\overline{g}}(y)),\quad x, y\in C^{\infty}(\mathbb{T}^{d}_{\theta}),
\end{eqnarray*}
where 
$$T_{\overline{g}}(y):=\sum\limits_{m\in\mathbb{Z}^d}\overline{g(m)}\widehat{y}(m)e_m^\theta,\quad y\in C^{\infty}(\mathbb{T}^{d}_{\theta}),$$
and  we used that $\tau_{\theta} \big(e_n^\theta 
   (e_m^\theta)^*\big)=0$ for $n\neq m.$
Since $C^{\infty}(\mathbb{T}^{d}_{\theta})$ is dense in $L^p(\mathbb{T}^d_\theta)$ (see \cite[Proposition 2.7. (ii), p. 21]{XXY}) for $1\leq p<\infty$ (with respect to the w*-topology for $p=\infty$), we have  $T^*_{g}=T_{\overline{g}}.$
\end{proof}

\section{Hausdorff-Young inequality and its inverse form }\label{H_Y_section}

In this section, we obtain the Hausdorff-Young inequality and its inverse form for the non-commutative torus. 

 \begin{thm}\label{thm3.1} (Hausdorff-Young inequality).   Let $1\leq{p}\leq2$  and $\frac{1}{p}+\frac{1}{p'}=1$. Then for any $x\in L^p(\mathbb{T}^d_\theta)$ we have
\begin{eqnarray}\label{additive3.1}
\|\widehat{x}\|_{\ell ^{p'}(\mathbb{Z}^d )}  \leq   \|x\|_{L^p(\mathbb{T}^d_\theta)}.
\end{eqnarray}
\end{thm}
\begin{proof}   We denote by $T$   the  
 operator defined by $T(x)=\{\widehat{x}(n)\}_{n\in\mathbb{Z}^d}$ for  $x\in L^p(\mathbb{T}^d_\theta),$ $1\leq p\leq 2.$ Clearly, $T$ is linear due to the linearity of the trace $\tau_\theta$ defined by (\ref{additive2.2}). Furthermore, by formulas  (\ref{additive2.5})  and (\ref{fourier-coef}) and using the  non-commutative H\"{o}lder inequality (see \cite[Theorem 1]{Sukochev}), we obtain  
\begin{eqnarray*} 
\|T(x)\|_{\ell^\infty(\mathbb{Z}^d )}&=&\|\widehat{x}\|_{\ell^\infty(\mathbb{Z}^d )} \overset{ \text{(\ref{fourier-coef})} }{=} \sup\limits_{n\in\mathbb{Z}^d}|\tau_\theta\left(x(e^\theta_n)^*\right)|  \\
&\leq&  \sup\limits_{n\in\mathbb{Z}^d}\|(e^\theta_n)^*\|_{L^\infty(\mathbb{T}^d_\theta)}\|x\|_{L^1(\mathbb{T}^d_\theta)} \overset{ \text{(\ref{additive2.5})} }{\leq}\|x\|_{L^1(\mathbb{T}^d_\theta)},  
\end{eqnarray*}  
 for all $n\in\mathbb{Z}^d.$
 Moreover, by the Plancherel identity (\ref{additive2.8}), we have    
\begin{eqnarray*} 
\|T(x)\|_{\ell^2(\mathbb{Z}^d)} =\|\widehat{x}\|_{\ell^2(\mathbb{Z}^d)}  \overset{ \text{(\ref{additive2.8})} }{=}\|x\|_{L^2 (\mathbb{T}^d_\theta )},\quad x\in L^2 (\mathbb{T}^d_\theta ).
\end{eqnarray*} 
Thus, $T$ is bounded from  $L^1(\mathbb{T}^d_\theta)$ into $\ell^\infty(\mathbb{Z}^d)$ and from $L^2(\mathbb{T}^d_\theta)$ into $\ell^2(\mathbb{Z}^d),$ with the operator norms at most 1. If $\eta=2/p',$ then $0\leq \eta\leq 1$ and we have $\frac{1}{p}=\frac{1-\eta}{1}+\frac{\eta}{2}$ and $\frac{1}{p'}=\frac{1-\eta}{\infty}+\frac{\eta}{2}.$
Hence, the assertion follows from the Riesz-Thorin Interpolation Theorem (see, \cite[Theorem 2.1]{PXu}).  
\end{proof} 

 As a direct consequence of Theorem \ref{thm3.1} we obtain the following Riemann-Lebesgue type lemma on $\mathbb T^d_\theta.$
\begin{lem}Let $x\in L^1(\mathbb{T}_{\theta}^d).$ Then $|\widehat{x}(m)| \to 0$ as $|m| \to\infty.$
\end{lem}
\begin{proof}  Let $p$ be a trigonometric polynomial defined by
$$
p = \sum_{m \in \mathbb{Z}^d} \widehat{p}(m) e_m^{\theta}, \qquad \widehat{p}(m)\in \mathbb{C}, \quad m \in \mathbb{Z}^d,
$$
that is, $\widehat{p}(m) = 0$ for all but finitely many $m \in \mathbb{Z}^d.$  
Since the involutive algebra of such polynomials is dense in $L^1(\mathbb{T}^{d}_{\theta})$ (see \cite{XXY}), for every operator $x \in L^1(\mathbb{T}^d_{\theta})$ and $\varepsilon>0,$ there exists a trigonometric polynomial $p$ such that  
$$
\|x - p\|_{L^1(\mathbb{T}^d_{\theta})} < \varepsilon. 
$$
Consequently, if  $|m|> \text{degree}(p),$ then $\widehat{p}(m) = 0,$ and thus by the Hausdorff-Young inequality (Theorem \ref{thm3.1} for $p=1$), we obtain
$$
|\widehat{x}(m)| = |\widehat{x}(m) - \widehat{p}(m)| =|\widehat{(x-p)}(m)|\leq \|x - p\|_{L^1(\mathbb{T}^d_{\theta})} < \varepsilon. 
$$
This completes the proof.
\end{proof}
The following is an analog of the classical Riesz-Fischer Theorem for Fourier Series on the non-commutative torus.
\begin{thm}(Riesz-Fischer Theorem for Fourier Series)\label{R-F-thm} Let $\{\xi(n)\}_{n\in\mathbb{Z}^d}$ be any complex valued sequence for which $\sum\limits_{n\in\mathbb{Z}^d}|\xi(n)|^2$ converges. Then there exists an operator $x\in L^2(\mathbb{T}^d_{\theta})$ with $\widehat{x}(n)=\xi(n)$ for all $n\in\mathbb{Z}^d.$ In particular,
$$\sum\limits_{n\in \mathbb{Z}^d}|\xi(n)|^2=\|x\|^{2}_{L^2(\mathbb{T}^d_{\theta})}.$$
\end{thm} 
\begin{proof}The proof is similar to the commutative case, therefore, is omitted (see \cite[Theorem 16.39, p. 250]{HS}).    
\end{proof}
The following result shows that inequality \eqref{additive3.1} remains true (with the reversed inequality) for $2\leq p\leq \infty.$
\begin{thm}\label{R-H-Y_ineq} If $2\leq p\leq\infty$ and $\xi=\{\xi(n)\}_{n\in \mathbb{Z}^{d}}\in \ell^{p'}(\mathbb{Z}^d),$ then the operator $x$ with $\widehat{x}(n)=\xi(n)$ for all $n\in \mathbb{Z}^{d}$ belongs to $L^{p}(\mathbb{T}^d_\theta),$ and we have
 \begin{eqnarray}\label{additive3.1'}
\|x\|_{L^p(\mathbb{T}^d_\theta)} \leq \|\widehat{x}\|_{\ell^{p'}(\mathbb{Z}^d )}
\end{eqnarray}
where $\frac{1}{p}+\frac{1}{p'}=1.$
 \end{thm}
\begin{proof}The idea of the proof is similar to that of Theorem \ref{thm3.1}. But, the main question is to define the operator correctly. Let us define the map $T_1$ by
$$T_1:\xi=\{\xi(n)\}_{n\in \mathbb{Z}^{d}}\mapsto x:= \sum\limits_{n\in\mathbb{Z}^d}\xi(n) e^\theta_n$$
which is well defined on $\ell^1(\mathbb{Z}^d)$ into $L^{\infty}(\mathbb{T}^d_\theta).$
On the other hand, Theorem \ref{R-F-thm} exhibits an isometry $T_2$ from $\ell^2(\mathbb{Z}^d)$ onto $L^2(\mathbb{T}^d_\theta)$ by $T_2 \xi=x$ with the property $\xi(n)=\widehat{x}(n)$ for all $n\in \mathbb{Z}^{d}.$ Therefore, if 
$\xi\in \ell^1(\mathbb{Z}^d),$ then $T_1 \xi$ and $T_2 \xi$ have the same Fourier coefficients and hence coincide. Since $\ell^1(\mathbb{Z}^d)$ is dense in $\ell^2(\mathbb{Z}^d),$ the operator $T_2$ is therefore the unique extension of $T_1$ to an isometry of $\ell^2(\mathbb{Z}^d)$ into $L^2(\mathbb{T}^d_\theta).$ This extension defines the Fourier coefficients with respect to $\mathbb{Z}^d.$ Interpolating as in Theorem \ref{thm3.1}, we obtain that if $\xi=\{\xi(n)\}_{n\in \mathbb{Z}^{d}}\in \ell^{p'}(\mathbb{Z}^d),$ then the operator $x$ with $\xi(n)=\widehat{x}(n)$ for all $n\in \mathbb{Z}^{d}$ belongs to $L^{p}(\mathbb{T}^d_\theta),$ and
$$\|x\|_{L^{p}(\mathbb{T}^d_\theta)}\leq \|\xi\|_{\ell^{p'}(\mathbb{Z}^d)}.$$
In other words, we obtain \eqref{additive3.1'}.
\end{proof}

\section{Hausdorff–Young–Paley Inequality}\label{sc4}

Here we give an analogue of the Hausdorff–Young–Paley inequality on $\mathbb{T}^d_\theta$ which is a generalization of the Hausdorff–Young inequality.
\begin{thm}\label{thm4.1} (Paley-type inequality). Let $1 < p \leq 2$ and let $h:\mathbb{Z}^d\to \mathbb{R}_{+}$ be a strictly positive function such that   
 \begin{eqnarray}\label{additive4.1}
 M_h:=\sup\limits_{\lambda>0}\lambda\sum\limits_{n\in\mathbb{Z}^d \atop  h(n)\geq\lambda}1<\infty.
 \end{eqnarray}
Then for any $x\in L^p(\mathbb{T}^d_\theta)$ we have the following estimate 
\begin{eqnarray}\label{additive4.2}
\left(\sum\limits_{n\in\mathbb{Z}^d} |\widehat{x}(n)|^p h^{2-p}(n) \right)^\frac{1}{p}\leq c_p M^\frac{2-p}{p}_h\|x\|_{L^p(\mathbb{T}^d_\theta)},
 \end{eqnarray}   
 where $c_p>0$ is a constant independent of $x.$
  \end{thm}
\begin{proof}Define a measure $\mu$ on $\mathbb{Z}^d$ such that 
$\mu(n):= h^2(n)>0$ for $n\in\mathbb{Z}^d$. Then it can be proved that the space $\ell^p(\mathbb{Z}^d, \mu)$ is a Banach space of all complex valued sequences $a = \{a(n)\}_{n\in\mathbb{Z}^d}$ with the norm
\begin{eqnarray*} 
\|a\|_{\ell^p(\mathbb{Z}^d, \mu)}:=\left(\sum\limits_{n\in\mathbb{Z}^d}|a(n)|^p h^2(n)\right)^\frac{1}{p}<\infty.
\end{eqnarray*}
Now, define an operator $A:L^p(\mathbb{T}^d_\theta)\to \ell^p(\mathbb{Z}^d, \mu)$ by the formula 
\begin{eqnarray*} 
Ax=\left\{\frac{\widehat{x}(n)}{h(n)}\right\}_{n\in\mathbb{Z}^d}.
\end{eqnarray*}
It is clear that $A$ is a sub-linear (or quasi-linear) operator. Next, we will show that $A$
is well-defined and bounded from $L^p(\mathbb{T}^d_\theta)$ to  $\ell^p(\mathbb{Z}^d, \mu)$ for $1 < p \leq 2$. In other words, we claim that we have the estimate (\ref{additive4.2}) with the condition (\ref{additive4.1}). First, we will show that $A$ is of weak type $(2,2)$ and of weak-type $(1,1)$. Recall that the distribution function $d_{Ax}(\lambda),$ $\lambda>0$ with respect
to the weight $h^2(n)>0$ is defined by
 \begin{eqnarray*} 
d_{Ax}(\lambda):=\mu\{n\in\mathbb{Z}^d: |(Ax)(n)|>\lambda\}=\sum\limits_{n\in\mathbb{Z}^d \atop |(Ax)(n)|>\lambda}h^2(n). 
\end{eqnarray*}

We show that 
\begin{eqnarray}\label{additive4.3}
d_{Ax}(\lambda)\leq c_2\frac{\|x\|^2_{L^2(\mathbb{T}^d_\theta)}}{\lambda^2 } \;\;\;\text{with}\;\;\;c_2=1,
\end{eqnarray}
 and 
\begin{eqnarray}\label{additive4.4}
d_{Ax}(\lambda)\leq \frac{c_1\|x\|_{L^1(\mathbb{T}^d_\theta)}}{\lambda}   \;\;\;\text{with}\;\;\;c_1=2M_h. 
\end{eqnarray}

Indeed,  we first prove inequality  (\ref{additive4.3}). It follows from the Chebyshev inequality and Plancherel identity (\ref{additive2.8}) that
 \begin{eqnarray*} 
\lambda^2d_{Ax}(\lambda)\leq \|Ax\|^2_{\ell^2(\mathbb{Z}^d, \mu)}
= \sum\limits_{n\in\mathbb{Z}^d} |\widehat{x}(n)|^2
\overset{ \text{(\ref{additive2.8}})}{=}\|x\|^2_{L^2(\mathbb{T}^d_\theta)}.
\end{eqnarray*} 
Thus, $A$ is of weak type $(2,2)$ with operator norm at most $c_2=1$.  
Now, let us show \eqref{additive4.4}. Applying formulas (\ref{additive2.5}) and (\ref{additive2.6}), and the non-commutative H\"{o}lder inequality, we obtain
\begin{eqnarray*}
\frac{|\widehat{x}(n)|}{h(n)} \overset{ \text{(\ref{additive2.6})}} {=} \frac{|\tau_\theta(x(e^\theta_n)^*)|}{h(n)}\leq\frac{\|x\|_{L^1(\mathbb{T}^d_\theta)}\|(e^\theta_n)^*\|_{L^\infty(\mathbb{T}^d_\theta)}}{h(n)}\overset{ \text{(\ref{additive2.5})}} {=}\frac{\|x\|_{L^1(\mathbb{T}^d_\theta)} }{h(n)}, \quad n\in\mathbb{Z}^d.
\end{eqnarray*} 
Therefore, we have
\begin{eqnarray*}
 \{n\in\mathbb{Z}^d: \frac{|\widehat{x}(n)|}{h(n)}>\lambda\} \subset
 \{n\in\mathbb{Z}^d: \frac{\|x\|_{L^1(\mathbb{T}^d_\theta)}}{h(n)}>\lambda\}
\end{eqnarray*} for any $\lambda>0.$
Consequently,
  \begin{eqnarray*}
 \mu\{n\in\mathbb{Z}^d: \frac{|\widehat{x}(n)|}{h(n)}>\lambda\} \leq
\mu\{n\in\mathbb{Z}^d: \frac{\|x\|_{L^1(\mathbb{T}^d_\theta)}}{h(n)}>\lambda\}
\end{eqnarray*}
for any $\lambda>0.$
Setting  $v:=\frac{\|x\|_{L^1(\mathbb{T}^d_\theta)}}{\lambda}$, we obtain 
\begin{equation}\label{additive4.5}
 \mu\{n\in\mathbb{Z}^d: \frac{|\widehat{x}(n)|}{h(n)}>\lambda\}
\leq\mu\{n\in\mathbb{Z}^d: \frac{\|x\|_{L^1(\mathbb{T}^d_\theta)}}{h(n)}>\lambda\}=\sum\limits_{n\in\mathbb{Z}^d \atop  h(n)\leq v}h^2(n).
\end{equation}
Let us estimate the right hand side. We now claim that
\begin{eqnarray}\label{additive4.6}
 \sum\limits_{n\in\mathbb{Z}^d \atop  h(n)\leq v}h^2(n)\leq 2v\cdot M_h .
\end{eqnarray}
Indeed, first we have
$$\sum\limits_{n\in\mathbb{Z}^d \atop  h(n)\leq v}h^2(n)=\sum\limits_{n\in\mathbb{Z}^d \atop  h(n)\leq v}\int\limits_0^{h^2(n)}dt\nonumber.$$
By interchanging the order of integration we obtain
$$\sum\limits_{n\in\mathbb{Z}^d \atop  h(n)\leq v}\int\limits_0^{h^2(n)}dt\nonumber=\int\limits_0^{v^2}dt\sum\limits_{n\in\mathbb{Z}^d \atop  t^\frac{1}{2}\leq h(n)\leq v}1\nonumber.$$
Further, by a substitution $t = s^2$ we have
$$\int\limits_0^{v^2}dt\sum\limits_{n\in\mathbb{Z}^d \atop  t^\frac{1}{2}\leq h(n)\leq v}1\nonumber=2\int\limits_0^{v}s ds\sum\limits_{n\in\mathbb{Z}^d \atop  s\leq h(n)\leq v}1\leq 2\int\limits_0^{v}s ds\sum\limits_{n\in\mathbb{Z}^d \atop  s\leq h(n)}1.$$
Since 
$$s \sum\limits_{n\in\mathbb{Z}^d \atop  s\leq h(n)}1\leq \sup_{s>0}s\sum\limits_{n\in\mathbb{Z}^d \atop  s\leq h(n)}1= M_{h},$$
and $M_{h}<\infty$ by assumption, it follows that
$$2\int\limits_0^{v}sds\sum\limits_{n\in\mathbb{Z}^d \atop  s\leq h(n)}1\leq 2 v\cdot M_{h}.$$
This proves the claim \eqref{additive4.6}.
Combining \eqref{additive4.5} and \eqref{additive4.6} we obtain (\ref{additive4.4}) which shows that $A$ is indeed of weak type $(1,1)$ with operator norm at most $c_1 = 2M_h$. Then, using the Marcinkiewicz interpolation theorem (see \cite[Theorem 1.3.1]{BL1976}, \cite[Theorem  5.2]{Dirk}) with $p_1 = 1,$ $p_2 = 2$
and $\frac{1}{p}=\frac{1-\eta}{1}+\frac{\eta}{2},$ we now obtain inequality (\ref{additive4.2}).  This completes the proof.  
\end{proof}
\begin{rem}Let $\theta=0$ and $d=1.$ If $h(n)=\frac{1}{1+|n|}, n\in\mathbb{Z},$ then $M_h<\infty$ in Theorem \ref{thm4.1}. Consequently, we obtain the classical Hardy-Littlewood inequality \cite{HL} (see also \cite[Theorem 2.1]{ANR}).
    
\end{rem}

As a consequence of the Paley-type inequality in Theorem \ref{thm4.1}, we obtain the following Hardy-Littlewood-type inequality on the non-commutative torus. We have very recently found, after we have proved, that this inequality indeed exists in \cite[Remark 7.3]{S-G} without proof. See also \cite{AMR}, where similar results are obtained for compact quantum groups. For the convenience of the reader, we give a new and simple proof, as a consequence of Theorem \ref{thm4.1}.

\begin{thm}\label{H-L_ineq} (Hardy-Littlewood-type inequality). Let $1 < p \leq 2$ and let $\varphi:\mathbb{Z}^d\to \mathbb{R}_{+}$ be a strictly positive
function such that 
$$\sum\limits_{n\in\mathbb{Z}^d}\frac{1}{\varphi^{\beta}(n)}<\infty \quad \text{for}\quad \text{some} \quad \beta>0.
$$
 Then for any $x\in L^p(\mathbb{T}^d_\theta)$ we have 
\begin{eqnarray}\label{H_L_ineq}
\left(\sum\limits_{n\in\mathbb{Z}^d} |\widehat{x}(n)|^p\varphi^{\beta(p-2)}(n) \right)^\frac{1}{p}\leq c_p \|x\|_{L^p(\mathbb{T}^d_\theta)},
 \end{eqnarray}   
 where $c_p>0$ is a constant independent of $x.$
  \end{thm}
  \begin{proof}By assumption, we have
  $$C:=\sum\limits_{n\in\mathbb{Z}^d}\frac{1}{\varphi^{\beta}(n)}<\infty \quad \text{for}\quad \text{some} \quad \beta> 0.$$
Thus,
$$C\geq \sum\limits_{n\in\mathbb{Z}^d \atop \varphi^{\beta}(n)\leq \frac{1}{s}}\frac{1}{\varphi^{\beta}(n)}\geq s\sum\limits_{n\in\mathbb{Z}^d \atop \varphi^{\beta}(n)\leq \frac{1}{s}}1=s\sum\limits_{n\in\mathbb{Z}^d \atop s\leq \frac{1}{\varphi^{\beta}(n)}}1,$$
where $s>0.$
Hence, taking the supremum with respect to $\lambda>0,$ we obtain
$$\sup_{s>0}s\sum\limits_{n\in\mathbb{Z}^d \atop s\leq \frac{1}{\varphi^{\beta}(n)}}1\leq C<\infty.$$
Then, applying Theorem \ref{thm4.1} for the function $h(n)=\frac{1}{\varphi^{\beta}(n)},$ $n\in\mathbb{Z}^{d},$
we obtain the desired result.
  \end{proof}
It can be seen that Theorem \ref{H-L_ineq} shows a necessary condition for the operator $x$ to belong to $L^p(\mathbb{T}^{d}_{\theta})$ for $1<p\leq2.$ However, this is not a sufficient condition. Thus, there is a natural question of finding a sufficient condition for $x$ (in terms of its Fourier coefficients) to belong to $L^p(\mathbb{T}^{d}_{\theta}).$  The following result answers this question.
\begin{thm}\label{R-H-L_ineq} (Inverse Hardy-Littlewood-type inequality). Let $2 \leq p <\infty$ be with $\frac{1}{p}+\frac{1}{p'}=1$ and let $\varphi:\mathbb{Z}^d\to \mathbb{R}_{+}$ be a strictly positive
function such that 
$$\sum\limits_{n\in\mathbb{Z}^d}\frac{1}{\varphi^{\beta}(n)}<\infty \quad \text{for}\quad \text{some} \quad \beta> 0.
$$ 
If $\sum\limits_{n\in\mathbb{Z}^d}\varphi^{\frac{\beta p(2-p')}{p'}}(n)|\widehat{x}(n)|^p<\infty,
$ then $x\in L^p(\mathbb{T}^{d}_{\theta})$ and we have
\begin{eqnarray}\label{R_H_L_ineq}
\|x\|_{L^p(\mathbb{T}^d_\theta)}\leq C_p \left(\sum\limits_{n\in\mathbb{Z}^d} \varphi^{\frac{\beta p(2-p')}{p'}}(n)|\widehat{x}(n)|^p \right)^\frac{1}{p},
 \end{eqnarray}   
   where $C_p>0$ is a constant independent of $x.$
  \end{thm}
  \begin{proof}By duality of $L^p(\mathbb{T}^{d}_{\theta})$ we have
  $$\|x\|_{L^p(\mathbb{T}^{d}_{\theta})}=\sup\{|\tau_{\theta}(xy^*)|:y\in L^{p'}(\mathbb{T}^{d}_{\theta}),\quad
  \|y\|_{L^{p'}(\mathbb{T}^{d}_{\theta})}=1\}.$$
  By \eqref{Plancerel-relation}, we have
 \begin{eqnarray*}\label{Parseval-relation}
  \tau_{\theta}(xy^*)=\sum_{n\in\mathbb{Z}^d}\widehat{x}(n)\overline{\widehat{y}(n)}, \,\ x\in L^p(\mathbb{T}^{d}_{\theta}), \,\ y\in L^{p'}(\mathbb{T}^{d}_{\theta}).
  \end{eqnarray*}
  Since $|\widehat{x}(n)\overline{\widehat{y}(n)}|\leq |\widehat{x}(n)||\overline{\widehat{y}(n)}|$ for all $n\in\mathbb{Z}^d,$ applying the H\"{o}lder inequality for any $y\in L^{p'}(\mathbb{T}^{d}_{\theta})$ with $\|y\|_{L^{p'}(\mathbb{T}^{d}_{\theta})}=1,$ we obtain
  \begin{eqnarray*}
\|x\|_{L^p(\mathbb{T}^d_\theta)}&=&\sup\{|\tau_{\theta}(xy^*)|:y\in L^{p'}(\mathbb{T}^{d}_{\theta}),\quad
  \|y\|_{L^{p'}(\mathbb{T}^{d}_{\theta})}=1\}\\
&=&\sup\left\{\left|\sum_{n\in\mathbb{Z}^d}\widehat{x}(n)\overline{\widehat{y}(n)}\right|:y\in L^{p'}(\mathbb{T}^{d}_{\theta}),\quad
  \|y\|_{L^{p'}(\mathbb{T}^{d}_{\theta})}=1\right\}\\
&\leq &\sup\left\{\sum_{n\in\mathbb{Z}^d}|\widehat{x}(n)||\overline{\widehat{y}(n)}|:y\in L^{p'}(\mathbb{T}^{d}_{\theta}),\quad
  \|y\|_{L^{p'}(\mathbb{T}^{d}_{\theta})}=1\right\}\\
&=&\sup\left\{\sum_{n\in\mathbb{Z}^d}|\widehat{x}(n)||\widehat{y}(n)|:y\in L^{p'}(\mathbb{T}^{d}_{\theta}),\quad
  \|y\|_{L^{p'}(\mathbb{T}^{d}_{\theta})}=1\right\} \\
  &=&\sup\left\{\sum_{n\in\mathbb{Z}^d}\varphi^{\frac{\beta (2-p')}{p'}}(n)|\widehat{x}(n)|\cdot\varphi^{\frac{\beta (p'-2)}{p'}}(n)|\widehat{y}(n)|:y\in L^{p'}(\mathbb{T}^{d}_{\theta}),\quad
  \|y\|_{L^{p'}(\mathbb{T}^{d}_{\theta})}=1\right\} \\
  &\leq&\sup_{y\in L^{p'}(\mathbb{T}^{d}_{\theta})\atop
  \|y\|_{L^{p'}(\mathbb{T}^{d}_{\theta})}=1}\left\{\left(\sum_{n\in\mathbb{Z}^d}\varphi^{\frac{\beta p(2-p')}{p'}}(n)|\widehat{x}(n)|^p\right)^{1/p}\cdot\left(\sum_{n\in\mathbb{Z}^d}\varphi^{\beta (p'-2)}(n)|\widehat{y}(n)|^{p'}\right)^{1/p'}
 \right\}.
\end{eqnarray*} 
Now applying Theorem \ref{H-L_ineq} with respect to $p',$ we get 
 \begin{eqnarray*}
\|x\|_{L^p(\mathbb{T}^d_\theta)}
  &\leq&\sup_{y\in L^{p'}(\mathbb{T}^{d}_{\theta})\atop
  \|y\|_{L^{p'}(\mathbb{T}^{d}_{\theta})}=1}\left\{\left(\sum_{n\in\mathbb{Z}^d}\varphi^{\frac{\beta p(2-p')}{p'}}(n)|\widehat{x}(n)|^p\right)^{1/p}\cdot\left(\sum_{n\in\mathbb{Z}^d}\varphi^{\beta (p'-2)}(n)|\widehat{y}(n)|^{p'}\right)^{1/p'}
 \right\}\\
 &\leq c_{p'} &\left(\sum_{n\in\mathbb{Z}^d}\varphi^{\frac{\beta p(2-p')}{p'}}(n)|\widehat{x}(n)|^p\right)^{1/p}\cdot\sup_{y\in L^{p'}(\mathbb{T}^{d}_{\theta})\atop
  \|y\|_{L^{p'}(\mathbb{T}^{d}_{\theta})}=1}\|y\|_{L^{p'}(\mathbb{T}^{d}_{\theta})}.
 \end{eqnarray*} 
Since $\|y\|_{L^{p'}(\mathbb{T}^{d}_{\theta})}=1,$ taking $C_p=c_{p'},$ we obtain the desired result.
  \end{proof}
\begin{rem} If $p=2,$ then both of the statements in Theorems \ref{H-L_ineq} and \ref{R-H-L_ineq} reduce to the Plancherel identity \eqref{additive2.8}.    
\end{rem}

The following result can be inferred from \cite[Corollary 5.5.2, p. 120]{BL1976} (or see \cite[Theorem  5.2]{Dirk} and \cite[Theorem 7.13.2, p. 511]{DdePS}).
\begin{prop}\label{weithg_inter}
  Let $ d\mu_1(n) = \mu_1(n)d\nu(n)$, $d\mu_2(n)=\mu_2(n)d\nu(n), n\in\mathbb{Z}^{d}.$ Suppose that $1\leq p,q_0, q_1 < \infty.$ If a continuous linear operator $A$ admits bounded extensions $A:L^p(\mathbb{T}^{d}_{\theta}) \rightarrow
\ell^{q_0} (\mathbb{Z}^{d},\mu_1)$  and $A:L^p(\mathbb{T}^{d}_{\theta}) \rightarrow \ell^{q_1}(\mathbb{Z}^{d},\mu_2),$ then there exists a bounded extension $A:L^p(\mathbb{T}^{d}_{\theta})\rightarrow  \ell^q(\mathbb{Z}^{d},\mu)$ where $0 < \eta < 1,$ $\frac{1}{q}=\frac{1-\eta }{q_0}+\frac{\eta}{q_1}$ and $\mu=\mu_1^\frac{q(1-\eta)}{q_0}\cdot\mu_2^\frac{q\eta }{q_1}$.
\end{prop}

By using this interpolation result between the Paley-type inequality (\ref{additive4.2})  in Theorem \ref{thm4.1} and the Hausdorff-Young inequality (\ref{additive3.1}) in Theorem \ref{thm3.2}, we obtain the following  result. 

\begin{thm}\label{thm4.2}(Hausdorff–Young–Paley inequality). Let $1 < p \leq r \leq p' < \infty$ with $\frac{1}{p}+\frac{1}{p'}=1.$ Let $h$ be given as in Theorem \ref{thm4.1}. Then we have
\begin{eqnarray}\label{additive4.7}
\left(\sum\limits_{n\in\mathbb{Z}^d} \Big(|\widehat{x}(n)|h(n)^{\frac{1}{r}-\frac{1}{p'}} \Big)^r \right)^\frac{1}{r}\leq c_{p,r,p'} M^{\frac{1}{r}-\frac{1}{p'}}_h\|x\|_{L^p(\mathbb{T}^d_\theta)},
 \end{eqnarray}  
 where $c_{p,r,p'}>0$ is a constant independent of $x.$
\end{thm}

\begin{proof}As in Theorem \ref{thm3.1}, let us consider a linear operator $ Ax:= \{\widehat{x}(n)\}_{n\in\mathbb{Z}^d}$ on $L^p(\mathbb{T}^d_\theta).$ Hence, applying the Paley-type inequality (\ref{additive4.2}) with $1<p\leq 2,$ we obtain
\begin{eqnarray}\label{additive4.8}
\left(\sum\limits_{n\in\mathbb{Z}^d} |\widehat{x}(n)|^p h^{2-p}(n) \right)^\frac{1}{p}\leq c_p M^\frac{2-p}{p}_h\|x\|_{L^p(\mathbb{T}^d_\theta)}.
 \end{eqnarray} 
 In other words, $A$ is bounded from $L^p(\mathbb{T}^d_\theta)$ to the weighted space $\ell^{p}(\mathbb{Z}^d,\mu_1)$
 with the weight $\mu_1(n):= h^{2-p}(n)>0$ for $n\in\mathbb{Z}^d$.
On the other hand, by the Hausdorff-Young inequality (\ref{additive3.1}) in Theorem \ref{thm3.1}, we have
\begin{eqnarray*}
\| \widehat{x}\|_{\ell^{p'}(\mathbb{Z}^d)}\leq \|x\|_{L^p(\mathbb{T}^d_\theta)},
\end{eqnarray*} 
where $1\leq{p}\leq2$ with $\frac{1}{p}+\frac{1}{p'}=1$.
This shows that $A$ is bounded from $L^p(\mathbb{T}^d_\theta)$ to $\ell^{p'}(\mathbb{Z}^d,\mu_2),$
where $\mu_2(n):=1$ for all $n\in \mathbb{Z}^d.$ By Proposition \ref{weithg_inter} we infer that $A:L^p(\mathbb{T}^d_\theta)\rightarrow \ell^r(\mathbb{Z}^d, \mu)$ is bounded for any $r$ such that $p \leq r \leq p',$  where the space $\ell^r(\mathbb{Z}^d, \mu)$ is defined as the space of sequences $a=\{a(n)\}_{n\in\mathbb{Z}^{d}}$ with the norm 
$$\|a\|_{\ell^r(\mathbb{Z}^d, \mu)}:=\left(\sum\limits_{n\in\mathbb{Z}^d} |a(n)|^r \mu(n) \right)^\frac{1}{r}$$
and $\mu$ is a positive sequence over $\mathbb{Z}^{d}$ to be determined. Let us compute $\mu$ in our setting. Indeed, fix $\eta\in (0,1)$ such that $\frac{1}{r}=\frac{1-\eta}{p}+\frac{\eta}{p'}.$ Then we have that $\eta=\frac{p-r}{r(p-2)}$ and by Proposition \ref{weithg_inter} with respect to $q=r,$ $q_0=p,$ and $q_1=p',$ we obtain 
\begin{eqnarray*}\label{weight}
\mu(n)=(\mu_{1}(n))^{\frac{r(1-\eta)}{p}}\cdot(\mu_{2}(n))^{\frac{r\eta}{p'}}=(h^{2-p}(n))^{\frac{r(1-\eta)}{p}}\cdot1^{\frac{r\eta}{p'}}=h^{1-\frac{r}{p'}}(n) 
\end{eqnarray*}
for all $n\in \mathbb{Z}^{d}$ and $\frac{2-p}{p}\cdot (1-\eta)=\frac{1}{r}-\frac{1}{p'}.$
Thus, 
$$\|A(x)\|_{\ell^{r}(\mathbb{Z}^d, \mu)}\leq c_{p,r} (M^\frac{2-p}{p}_h)^{1-\eta}\|x\|_{L^p(\mathbb{T}^d_\theta)}=  M^{\frac{1}{r}-\frac{1}{p'}}_h  \|x\|_{L^p(\mathbb{T}^d_\theta)}, \quad x\in L^p(\mathbb{T}^d_\theta).$$
This completes the proof.
\end{proof}

\section{H\"ormander multiplier theorem on quantum tori}\label{sc5}

In this section, we are concerned with the question of what assumptions on the symbol $g:\mathbb{Z}^d\to \mathbb{C}$ and parameters $1\leq p,q \leq \infty$ guarantee that $T_g$ (see \eqref{additive2.9}) has a bounded extension from $L^p (\mathbb{T}^d_\theta )$  to $L^q (\mathbb{T}^d_\theta ).$ The following result is a non-commutative analogue of the H\"ormander's result \cite[Theorem 1.11]{Hor} (see also \cite[p. 303]{E} for periodic version of the H\"ormander's theorem) on $\mathbb{T}^d_\theta.$

\begin{thm}\label{thm5.1}(H\"ormander multiplier theorem). Let $1 < p \leq 2 \leq q < \infty$ and let $g:\mathbb{Z}^d\to \mathbb{C}$ be a measurable function satisfying $$\sup\limits_{\lambda>0}\lambda\left(\sum\limits_{n\in\mathbb{Z}^d \atop  |g(n)|\geq\lambda}1\right)^{\frac{1}{p}-\frac{1}{q}}<\infty.$$ 
Then Fourier multiplier $T_g$ with the symbol $g$ has a bounded extension from $L^p(\mathbb{T}^d_\theta)$ to $L^q(\mathbb{T}^d_\theta),$ and we have
\begin{eqnarray}\label{additive5.1}
\|T_g\|_{L^p(\mathbb{T}^d_\theta) \rightarrow L^q(\mathbb{T}^d_\theta)}\leq c_{p,q} \sup\limits_{\lambda>0}\lambda\left(\sum\limits_{n\in\mathbb{Z}^d \atop  |g(n)|\geq\lambda}1\right)^{\frac{1}{p}-\frac{1}{q}},
 \end{eqnarray} 
 where $c_{p,q}>0$ is a constant independent of $x.$
\end{thm}

\begin{proof} By duality it is sufficient to study two cases: $1<p\leq q'\leq 2$ and $1<q'\leq p \leq 2,$ where $1=\frac{1}{q}+\frac{1}{q'}.$

First, we consider the case $1<p\leq q'\leq 2,$ where $1=\frac{1}{q}+\frac{1}{q'}.$ By (\ref{additive2.4}) and (\ref{fourier-coef}) and  (\ref{additive2.9}) we have 
\begin{eqnarray}\label{additive5.2}
\widehat{T_g(x)}(n)&\overset{ \text{(\ref{fourier-coef}})}{=}&\tau_\theta\left(T_g(x)(e^\theta_n)^*\right)\nonumber\\
&\overset{ \text{(\ref{additive2.9}})}{=}&\tau_\theta\left(\sum\limits_{m\in\mathbb{Z}^d}g(m)\widehat{x}(m)e^\theta_m(e^\theta_n)^*\right)\nonumber\\
&=&\sum\limits_{m\in\mathbb{Z}^d}g(m)\widehat{x}(m)\tau_\theta\left(e^\theta_m(e^\theta_n)^*\right)\nonumber\\
&\overset{ \text{(\ref{additive2.4}})}{=}&g(n)\widehat{x}(n), \quad\forall n\in \mathbb{Z}^{d}. 
\end{eqnarray} 
Then, it follows from Theorem \ref{R-H-Y_ineq} and (\ref{additive5.2})  that  
\begin{eqnarray}\label{additive5.3}
\|T_g(x)\|_{L^q\left(\mathbb{T}^d_\theta\right)} \overset{ \text{(\ref{additive3.1'})}}{\leq} \|\widehat{T_g(x)}\|_{\ell^{q'}\left(\mathbb{Z}^d\right)} 
 \overset{ \text{(\ref{additive5.2}})}{=} \|g\cdot\widehat{x}\|_{\ell^{q'}\left(\mathbb{Z}^d\right)},  
  \end{eqnarray}
for all $x\in L^p(\mathbb{T}^d_\theta).$

Therefore, if we set $q':=r$ and $\frac{1}{s}:=\frac{1}{p}-\frac{1}{q}=\frac{1}{q'}-\frac{1}{p'},$ then for $h(n):=|g(n)|^{s}, n\in \mathbb{Z}^{d},$ we are in a position to apply the Hausdorff-Young-Paley inequality in Theorem \ref{thm4.2}. In other words, we obtain
 \begin{eqnarray}\label{additive5.4}
\left(\sum\limits_{n\in\mathbb{Z}^d} \Big(|\widehat{x}(n)|\cdot|g(n)|\Big)^{q'}\right)^\frac{1}{q'}\lesssim M^\frac{1}{s}_{|g|^{s}}\|x\|_{L^p\left(\mathbb{T}^d_\theta\right)}
 \end{eqnarray} 
for any $x\in L^p(\mathbb{T}^d_\theta).$ 
Let us study $M^\frac{1}{s}_{|g|^{s}}$ separately. Indeed, by definition 
$$M^\frac{1}{s}_{|g|^{s}}:=\left(\sup\limits_{\lambda>0}\lambda\sum\limits_{n\in\mathbb{Z}^d \atop  |g(n)|^{s}\geq\lambda}1\right)^{\frac{1}{s}}=\left(\sup\limits_{\lambda>0}\lambda\sum\limits_{n\in\mathbb{Z}^d \atop  |g(n)|\geq\lambda^{\frac{1}{s}}}1\right)^{\frac{1}{s}}=\left(\sup\limits_{\lambda>0}\lambda^{s}\sum\limits_{n\in\mathbb{Z}^d \atop  |g(n)|\geq\lambda}1\right)^{\frac{1}{s}}.$$ Since $\frac{1}{s}:=\frac{1}{p}-\frac{1}{q},$ it follows that
\begin{eqnarray}\label{M_g_estimate}
M^\frac{1}{s}_{|g|^{s}}&=&\left(\sup\limits_{\lambda>0}\lambda^{s}\sum\limits_{n\in\mathbb{Z}^d \atop  |g(n)|\geq\lambda}1\right)^{\frac{1}{p}-\frac{1}{q}}\nonumber\\&=&\sup\limits_{\lambda>0}\lambda^{s\big(\frac{1}{p}-\frac{1}{q}\big)}\left(\sum\limits_{n\in\mathbb{Z}^d \atop  |g(n)|\geq\lambda}1\right)^{\frac{1}{p}-\frac{1}{q}}\nonumber\\&=&\sup\limits_{\lambda>0}\lambda\left(\sum\limits_{n\in\mathbb{Z}^d \atop  |g(n)|\geq\lambda}1\right)^{\frac{1}{p}-\frac{1}{q}}.
 \end{eqnarray}

Hence, combining (\ref{additive5.3}), (\ref{additive5.4}), and \eqref{M_g_estimate} we obtain
\begin{eqnarray*}
\|T_g(x)\|_{L^q(\mathbb{T}^d_\theta)}  
 &\overset{\text{(\ref{additive5.3})}}{\lesssim}&\left(\sum\limits_{n\in\mathbb{Z}^d} \Big(|\widehat{x}(n)|\cdot|g(n)|\Big)^{q'}\right)^\frac{1}{q'}  \\&\overset{ \text{(\ref{additive5.4})}} {\lesssim}& M^\frac{1}{s}_{|g|^{s}}\|x\|_{L^p\left(\mathbb{T}^d_\theta\right)} \\&\overset{ \text{\eqref{M_g_estimate}}}{=} & \sup\limits_{\lambda>0}\lambda\left(\sum\limits_{n\in\mathbb{Z}^d \atop  |g(n)|\geq\lambda}1\right)^{\frac{1}{p}-\frac{1}{q}}\|x\|_{L^p(\mathbb{T}^d_\theta)},  
\end{eqnarray*}
for $1 < p \leq q'\leq 2$ and  $x\in L^p(\mathbb{T}^d_\theta).$

Next, we consider the case $q'\leq p\leq2$ so that $p'\leq (q')'=q,$ where $1=\frac{1}{q}+\frac{1}{q'}$ and $1=\frac{1}{p}+\frac{1}{p'}.$ Thus, the $L^p$-duality (see Lemma \ref{Duality}) yields that $T^{*}_{g}=T_{\overline{g}}$ and
\begin{eqnarray*}
 \|T_g\|_{L^p(\mathbb{T}^d_\theta) \rightarrow L^q(\mathbb{T}^d_\theta)}= \|T_{\overline{g}}\|_{L^{q'}(\mathbb{T}^d_\theta) \rightarrow L^{p'}(\mathbb{T}^d_\theta)}.  
\end{eqnarray*}
The symbol of the adjoint operator $T^{*}_{g}$ equals to $\overline{g}$ and
obviously we have $|g(n)|=|\overline{g(n)}|$ for all $n\in\mathbb{Z}^{d}.$ Set $\frac{1}{p}-\frac{1}{q}=\frac{1}{s}=\frac{1}{q'}-\frac{1}{p'}.$ Hence, by repeating the argument in the previous case we have
\begin{eqnarray*}
\|T_{\overline{g}}(x)\|_{L^{p'}(\mathbb{T}^d_\theta)}  
 &\lesssim& \sup\limits_{\lambda>0}\lambda\left(\sum\limits_{n\in\mathbb{Z}^d \atop  |\overline{g(n)}|\geq\lambda}1\right)^{\frac{1}{q'}-\frac{1}{p'}}\|x\|_{L^{q'}(\mathbb{T}^d_\theta)}\\
 &=& \sup\limits_{\lambda>0}\lambda\left(\sum\limits_{n\in\mathbb{Z}^d \atop  |g(n)|\geq\lambda}1\right)^{\frac{1}{p}-\frac{1}{q}}\|x\|_{L^{q'}(\mathbb{T}^d_\theta)}.  
\end{eqnarray*}
In other words, we have 
$$\|T_{\overline{g}}\|_{L^{q'}(\mathbb{T}^d_\theta) \rightarrow L^{p'}(\mathbb{T}^d_\theta)}\lesssim
\sup\limits_{\lambda>0}\lambda\left(\sum\limits_{n\in\mathbb{Z}^d \atop  |g(n)|\geq\lambda}1\right)^{\frac{1}{p}-\frac{1}{q}}.$$ Combining both cases, we obtain
$$\|T_{g}\|_{L^{p}(\mathbb{T}^d_\theta) \rightarrow L^{q}(\mathbb{T}^d_\theta)}\lesssim
\sup\limits_{\lambda>0}\lambda\left(\sum\limits_{n\in\mathbb{Z}^d \atop  |g(n)|\geq\lambda}1\right)^{\frac{1}{p}-\frac{1}{q}}$$
for all $1 < p \leq 2 \leq q < \infty.$ This concludes the result.
\end{proof}

\begin{rem}Define the partial differentiation operators $\partial^{\theta}_j, j = 1,..., d,$ by the formula
\begin{equation}\label{partial diff}
\partial^{\theta}_j(e_{n}^{\theta}) =in_{j}e_{n}^{\theta}, \quad n = (n_1,...,n_d )\in \mathbb{Z}^{d}.
\end{equation}
Every partial derivation $\partial^{\theta}_j$ can be viewed as a densely defined closed (unbounded) operator
on $L^2(\mathbb{T}_{\theta}^{d})$, whose adjoint is equal to $-\partial^{\theta}_j.$ 
Let $\Delta_{\theta}=(\partial_1^{\theta})^2+\cdots+(\partial_d^{\theta})^2$ be the Laplacian. Then $\Delta_{\theta} = - ((\partial_1^{\theta})^* \partial_1^{\theta} + \cdots +(\partial_d^{\theta})^* \partial_d^{\theta}),$ so $-\Delta_{\theta}$ is a positive operator on $L^2(\mathbb{T}^{d}_\theta)$ with spectrum equal to $\{ |n|^2  : n \in \mathbb{Z}^d\}$ (see \cite{MSQ2019},\cite {XXY}). 
Note that when $\theta=0,$ $-\Delta_{\theta}$ reduces
to the classical Laplacian divided by $(2\pi)^2.$
Let us define the heat semigroups
of operators on $L^2(\mathbb{T}_{\theta}^{d})$ by 
$t\mapsto e^{t\Delta_{\theta}}.$
This operator can be defined as the Fourier multiplier. Observe that by definition (see \eqref{additive2.9}), it is an operator of the form $T_g$ with the symbol $g(n)=e^{-t|n|^2}, \,n \in \mathbb{Z}^d, \, t>0.$ Then, a simple application of Theorem \ref{thm5.1} gives us the $L^p -L^q$ properties of the propagator and  the time asymptotics of its non-commutative Lorentz space norm with the restriction $1<p\leq 2\leq q<\infty.$
However, a stronger result for all $1\leq p< q\leq\infty$ was established in \cite[Lemma 6.1]{XXY}.
\end{rem}

\section{Sobolev and logarithmic Sobolev inequalities on quantum tori}\label{sc8}
In this section, we obtain Sobolev and logarithmic Sobolev type inequalities on quantum tori. 
Let $\partial^{\theta}_j, j = 1,\cdots, d,$ be differentiation operators defined by \eqref{partial diff} in the previous section. We define the gradient $\nabla_{\mathbb{T}_{\theta}^{d}}$ by the formula
$$\nabla_{\theta}=(\partial^{\theta}_1,\cdots,\partial^{\theta}_d).$$
For a multi-index $\alpha=(\alpha_1,...,\alpha_d),$ define
$$\nabla^{\alpha}_{\theta}=(\partial_1^{\theta})^{\alpha_1}\cdots(\partial_d^{\theta})^{\alpha_d},$$
which is also considered as a linear  operator on $L^2(\mathbb{T}_{\theta}^{d}).$

\begin{definition} (see, \cite[Definition 2.6. (ii)]{XXY})
Let  $s\in\mathbb{R}$ and $1\leq p<\infty$. 
Then the potential (or fractional) Sobolev space of order $s\in\mathbb{R}$ is defined to be 
$$
H_p^s(\mathbb{T}^d_\theta)=\{x\in \mathcal{D}'(\mathbb{T}^d_\theta):  (1- \Delta_\theta)^\frac{s}{2}{x}\in L^p(\mathbb{T}^d_\theta) \}, 
$$
 equipped with the norm
$$ 
\|x\|_{H_p^s(\mathbb{T}^d_\theta)}=\|(1- \Delta_\theta)^\frac{s}{2}{x}\|_{L^p(\mathbb{T}^d_\theta)}, 
$$
 where \begin{equation}\label{laplacian}\Delta_{\theta}=(\partial_1^{\theta})^2+\cdots+(\partial_d^{\theta})^2
 \end{equation} is the Laplacian.
If $p=2,$ then as usual, we denote $H_2^s(\mathbb{T}^d_\theta)=:H^s(\mathbb{T}^d_\theta).$ If $k\in\mathbb{N},$ then the corresponding Sobolev space $W_{p}^{k}(\mathbb{T}^d_\theta)$ is defined to be the set of $x\in L^{p}(\mathbb{T}^d_\theta)$ with $\nabla^{\alpha}_{\theta}x\in L^{p}(\mathbb{T}^d_\theta)$ for all $|\alpha|_1 \leq k,$ where $|\alpha|_1=\alpha_1+...+\alpha_d.$ The $W_{p}^{k}$ norm is the sum of the $L^{p}$ norms of $\nabla^{\alpha}_{\theta}x$ for all $0\leq |\alpha|_1 \leq k,$ that is 
$$
\|x\|_{W_{p}^{k}(\mathbb{T}^d_\theta)}=\left(\sum_{0\leq |\alpha|_1 \leq k}\|\nabla^{\alpha}_{\theta}x\|^p_{L^{p}(\mathbb{T}^d_\theta)}\right)^{1/p}.$$
\end{definition}
The above definitions of Sobolev spaces on quantum tori and their main properties including embeddings were studied in \cite{XXY}. 

 As a consequence of Theorem \ref{thm5.1} we obtain the following Sobolev type embedding.  
\begin{thm}\label{sobolev-embed}
  Let $1 < p \leq 2 \leq q < \infty.$ Then the inequality  
$$\|x\|_{L^{q}(\mathbb{T}^d_\theta)}\leq c_{p,q,s} \|x\|_{H_p^s(\mathbb{T}^d_\theta)},
$$
where $c_{p,q,s}>0$ is a constant independent of $x,$ holds true provided that
$s\geq d\left(\frac{1}{p}-\frac{1}{q}\right).$ 
\end{thm}
\begin{proof} For $s>0,$ let us consider $g(n)=(1+|n|^2)^{- s/2}, \, n\in\mathbb{Z}^d.$ Take $T_g=(1-\Delta_{\theta})^{- s/2}.$ In order to use Theorem \ref{thm5.1}, we need first to calculate the right hand side of \eqref{additive5.1}. For this, we have
\begin{eqnarray*}
\sup\limits_{\lambda>0}\lambda\left(\sum\limits_{n\in\mathbb{Z}^d \atop  |g(n)|\geq\lambda}1\right)^{\frac{1}{p}-\frac{1}{q}}&=& \sup\limits_{ \lambda>0}\lambda\left(\sharp\{n\in\mathbb{Z}^d: \frac{1}{(1+|n|^2)^{s/2}}\geq \lambda\}\right)^{\frac{1}{p}-\frac{1}{q}}.
\end{eqnarray*}
Since 
\begin{eqnarray*}
\sharp\{n\in\mathbb{Z}^d: \frac{1}{(1+|n|^2)^{s/2}}\geq \lambda\}
= \begin{cases} 0, \,\ \text{if} \,\ 1<\lambda<\infty, \\
\sharp\{n\in\mathbb{Z}^d: |n|\leq \sqrt{\lambda^{-2/s}-1}\},\,\ \text{if} \,\ 0<\lambda\leq 1,
    \end{cases}
\end{eqnarray*}
it follows that
\begin{eqnarray*}
\sup\limits_{\lambda>0}\lambda\left(\sum\limits_{n\in\mathbb{Z}^d \atop  |g(n)|\geq\lambda}1\right)^{\frac{1}{p}-\frac{1}{q}}&=& \sup\limits_{0<\lambda\leq 1}\lambda\left(\sharp\{n\in\mathbb{Z}^d: \frac{1}{(1+|n|^2)^{s/2}}\geq \lambda\}\right)^{\frac{1}{p}-\frac{1}{q}}\\
&=& \sup\limits_{0<\lambda\leq 1}\lambda\left(\sharp\{n\in\mathbb{Z}^d: |n|\leq \sqrt{\lambda^{-2/s}-1}\}\right)^{\frac{1}{p}-\frac{1}{q}}\\
&\leq&\sup\limits_{0<\lambda\leq 1}\lambda\left(\sharp\{n\in\mathbb{Z}^d: |n|\leq \lambda^{-1/s}\}\right)^{\frac{1}{p}-\frac{1}{q}}\\
&\lesssim&\sup\limits_{0<\lambda\leq 1}\lambda^{1-\frac{d}{s}(\frac{1}{p}-\frac{1}{q})}<\infty,
\end{eqnarray*}
whenever $s\geq d\left(\frac{1}{p}-\frac{1}{q}\right),$ where  $\sharp$  is a counting measure on $\mathbb{Z}^d.$ 
Thus, the function $g$ satisfies the assumptions of Theorem \ref{thm5.1}, and consequently, we have \begin{eqnarray*}
\|x\|_{L^{q}(\mathbb{T}^d_\theta)}&=& \|(1-\Delta_{\theta})^{-s/2}(1-\Delta_{\theta})^{s/2}x\|_{L^{p}(\mathbb{T}^d_\theta)}\\
&\overset{\eqref{additive5.1}}{\leq}& c_{p,q,s} \sup\limits_{\lambda>0}\lambda\left(\sum\limits_{n\in\mathbb{Z}^d \atop  |g(n)|\geq\lambda}1\right)^{\frac{1}{p}-\frac{1}{q}}\|(1-\Delta_{\theta})^{s/2}x\|_{L^{p}(\mathbb{T}^d_{\theta})}\\
&\leq &c_{p,q,s}\|x\|_{H_p^s(\mathbb{T}^d_\theta)}
\end{eqnarray*}
for $s\geq d\left(\frac{1}{p}-\frac{1}{q}\right)$ and $x\in H_p^s(\mathbb{T}^d_\theta),$
thereby completing the proof.
\end{proof}

We also obtain the following a simple consequence of the non-commutative H\"{o}lder inequality.
\begin{lem}\label{inter-lem}Let $1\leq p \leq r \leq q\leq \infty$   be such that $\frac{1}{r}=\frac{\eta}{p}+\frac{1-\eta}{q}$ for some $\eta\in(0,1).$ Then for any $x\in L^q(\mathbb{T}^d_\theta),$ we have
\begin{equation}\label{inter-estimate}
\|x\|_{L^r(\mathbb{T}^d_\theta)}\leq \|x\|^{\eta}_{L^p(\mathbb{T}^d_\theta)}\|x\|^{1-\eta}_{L^q(\mathbb{T}^d_\theta)}.
 \end{equation}
\end{lem}
\begin{proof}The proof is a simple consequence of the non-commutative H\"{o}lder inequality \cite[Theorem 1]{Sukochev}, therefore, is omitted.
\end{proof}

\begin{lem}\label{log-holder}(Logarithmic H\"{o}lder inequality)
 Let $1\leq p< q< \infty.$ Then for any $0\neq x\in L^p(\mathbb{T}^d_\theta)$ we have    \begin{equation}\label{log-estimate}
\tau_{\theta}\left(\frac{|x|^p}{\|x\|^p_{L^p(\mathbb{T}^d_\theta)}}\log \Big(\frac{|x|^p}{\|x\|^p_{L^p(\mathbb{T}^d_\theta)}}\Big)\right)\leq \frac{q}{q-p}\log \Big(\frac{\|x\|^p_{L^q(\mathbb{T}^d_\theta)}}{\|x\|^p_{L^p(\mathbb{T}^d_\theta)}}\Big).
 \end{equation}
\end{lem}
\begin{proof} The proof is a modification of the proof in \cite{CAM2023}.  Let us consider the 
function 
$$f(\frac{1}{r})=\log (\|x\|_{L^r(\mathbb{T}^d_\theta)}), \quad r>0.$$
By Lemma \ref{inter-lem} we have
\begin{eqnarray*}f(\frac{1}{r})&=&\log (\|x\|_{L^r(\mathbb{T}^d_\theta)})\leq \log (\|x\|^{\eta}_{L^p(\mathbb{T}^d_\theta)}\|x\|^{1-\eta}_{L^q(\mathbb{T}^d_\theta)})\\
&=&\log (\|x\|^{\eta}_{L^p(\mathbb{T}^d_\theta)})+\log (\|x\|^{1-\eta}_{L^q(\mathbb{T}^d_\theta)})\\
&=&\eta f(\frac{1}{p})+(1-\eta) f(\frac{1}{q}), 
\end{eqnarray*}
with $r>1$ and $\eta\in (0,1)$ such that $\frac{1}{r}=\frac{\eta}{p}+\frac{1-\eta}{q}.$
This shows that the function $f$ is convex. Since
$$f(s)=s\log (\tau_{\theta}(|x|^{1/s})),$$ taking derivative with respect to $s$ (see \cite[Lemma 4.2]{Xiong}), we obtain
\begin{eqnarray*}f'(s)&=&\log (\tau_{\theta}(|x|^{1/s}))+s\Big(\log (\tau_{\theta}(|x|^{1/s}))\Big)'_{s}\\
&=&\log (\tau_{\theta}(|x|^{1/s}))+s\frac{\Big(\tau_{\theta}(|x|^{1/s})\Big)'_{s}}{\tau_{\theta}(|x|^{1/s})}\\
&=&\log (\tau_{\theta}(|x|^{1/s}))-\frac{1}{s}\frac{\tau_{\theta}(|x|^{1/s}\log(|x|))}{\tau_{\theta}(|x|^{1/s})}. 
\end{eqnarray*}
On the other hand, since $f$ is a differentiable convex function of $s,$ it follows that
$$f'(s)\geq \frac{f(s_0)-f(s)}{s_0-s}, \quad s_0>s>0.$$
Taking $s=\frac{1}{p}$ and $s_0=\frac{1}{q},$ we have
\begin{equation}\label{almost-log-holder}
    p\frac{\tau_{\theta}(|x|^{p}\log(|x|))}{\tau_{\theta}(|x|^{p})}-\log (\tau_{\theta}(|x|^{p}))
\leq\frac{pq}{q-p}\log \Big(\frac{\|x\|_{L^q(\mathbb{T}^d_\theta)}}{\|x\|_{L^p(\mathbb{T}^d_\theta)}}\Big).
\end{equation}
But, by linearity of $\tau_{\theta}$ the left hand side of the previous inequality is represented as follows
\begin{eqnarray*}&p&\frac{\tau_{\theta}(|x|^{p}\log(|x|))}{\tau_{\theta}(|x|^{p})}-\log (\tau_{\theta}(|x|^{p}))=p\frac{\tau_{\theta}(|x|^{p}\log(|x|))}{\tau_{\theta}(|x|^{p})}-\frac{\tau_{\theta}(|x|^{p})\cdot\log (\tau_{\theta}(|x|^{p}))}{\tau_{\theta}(|x|^{p})}\\
&=&\frac{\tau_{\theta}(|x|^{p}\log(|x|^p))}{\tau_{\theta}(|x|^{p})}-\frac{\tau_{\theta}\Big(|x|^{p}\log (\|x\|^{p}_{L^p(\mathbb{T}^d_\theta)})\Big)}{\tau_{\theta}(|x|^{p})}\\
&=&\frac{\tau_{\theta}(|x|^{p}\log(|x|^p))-\tau_{\theta}\Big(|x|^{p}\log (\|x\|^{p}_{L^p(\mathbb{T}^d_\theta)})\Big)}{\tau_{\theta}(|x|^{p})}\\
&=&\frac{\tau_{\theta}\Big(|x|^{p}\Big(\log(|x|^p)-\log (\|x\|^{p}_{L^p(\mathbb{T}^d_\theta)})\Big)\Big)}{\tau_{\theta}(|x|^{p})}=\tau_{\theta}\left(\frac{|x|^p}{\|x\|^p_{L^p(\mathbb{T}^d_\theta)}}\log \Big(\frac{|x|^p}{\|x\|^p_{L^p(\mathbb{T}^d_\theta)}}\Big)\right).
\end{eqnarray*}
Hence, the assertion follows from \eqref{almost-log-holder}.
\end{proof}

As an application of Theorem \ref{sobolev-embed} for the case $1 < p < 2$ and applying \cite[Theorem 6.6. (i), p. 70]{XXY} for $2\leq p<\infty,$ we obtain the following logarithmic type Sobolev inequality. We have found recently in \cite{JLZZ} that the logarithmic Sobolev inequality is valid for an arbitrary hypercontractive semigroup acting on a noncommutative probability space (see also \cite{Xiong}).
Another type of logarithmic Sobolev inequality was obtained in \cite{Lee} on quantum tori when $p=2.$
\begin{thm}(Logarithmic Sobolev inequality)\label{sobolev-ineq} Let $s>0$ and $1<p<\infty$ be such that $d>sp.$
Then for any $0\neq x\in H_p^s(\mathbb{T}^d_\theta)$  we have the fractional
logarithmic Sobolev  inequality
\begin{equation}\label{sobolev-log-estimate}
\tau_{\theta}\left(\frac{|x|^p}{\|x\|^p_{L^p(\mathbb{T}^d_\theta)}}\log \Big(\frac{|x|^p}{\|x\|^p_{L^p(\mathbb{T}^d_\theta)}}\Big)\right)\leq \frac{d}{sp}\log \Big(C_{p,d}\frac{\|x\|^p_{H_p^s(\mathbb{T}^d_\theta)}}{\|x\|^p_{L^p(\mathbb{T}^d_\theta)}}\Big), 
 \end{equation}
where $C_{p,d}>0$ is a constant independent of $x.$
\end{thm}
\begin{proof}
By Theorem \ref{sobolev-embed} for  $1 < p \leq 2\leq p^*<\infty$ and by \cite[Theorem 6.6. (i), p. 70]{XXY} for other cases, we have $H_p^s(\mathbb{T}^d_\theta)\subset L^{p^*}(\mathbb{T}^d_\theta)$ with $1<p<p^*<\infty$ and $s>0$ such that $s=\frac{d}{p}-\frac{d}{p^*}.$ This means that there exists a constant $C_{p,p^*,d}>0$ such that
$$\|x\|_{L^{p^*}(\mathbb{T}^d_\theta)}\leq C_{p,p^*,d}\|x\|_{H_p^s(\mathbb{T}^d_\theta)}, \,\ x\in H_p^s(\mathbb{T}^d_\theta). $$
Taking $q:=p^*=\frac{dp}{d-sp}$ in Lemma \ref{log-holder} and since $\frac{p^*}{p^* -q}=\frac{d}{sp},$ the assertion follows from \eqref{log-estimate}.
\end{proof}

\subsection{Nash type inequality on quantum tori}
In this subsection, we prove the Nash inequality on quantum tori. In the classical case, this inequality is applied as a main tool in computing the decay rate for the heat equation for the sub-Laplacian. For more details about these inequalities and their applications to PDEs in Lie groups we refer the reader to \cite{CAM2023}.
\begin{thm}\label{Nash-inequality}(Nash type inequality). Let $x\in H^{1}_{2}(\mathbb{T}^d_\theta)$ and $d>2$. Then  we have 
 \begin{eqnarray}\label{Nash's-inequality}
  \|x\|^{1+\frac{2}{d}}_{L^{2}(\mathbb{T}^d_\theta)} \leq C_{1,d}\| x\|_{H^{1}_{2}(\mathbb{T}^d_\theta)}\|x\|^{\frac{2}{d}}_{L^{1}(\mathbb{T}^d_\theta)},  
 \end{eqnarray}   
  where $C_{1,d}>0$ is a constant independent of $x.$
\end{thm}

\begin{proof} 
By the Sobolev inequality in Theorem \ref{sobolev-embed}, we get
\begin{eqnarray}\label{s-inequality}
  \|x\|_{L^{2^*}(\mathbb{T}^d_\theta)} \leq C_{1,d} \|x\|_{H^{1}_{2}(\mathbb{T}^d_\theta)}, 
 \end{eqnarray} 
 where $C_{1,d}>0$ is a constant independent of $x$.  

By Lemma \ref{inter-lem} with 
 $r=2$, $q=2^* = \frac{2d}{d-2},$  and $p=1,$ we obtain
\begin{eqnarray}\label{interpolation-inequality}
  \|x\|_{L^{2}(\mathbb{T}^d_\theta)} \leq \|x\|^{\eta}_{L^{1}(\mathbb{T}^d_\theta)}\|x\|^{1-\eta}_{L^{2^*}(\mathbb{T}^d_\theta)},
 \end{eqnarray} 
where
\begin{eqnarray}\label{parametr}
\eta:=\frac{\frac{1}{r}-\frac{1}{q}}{\frac{1}{p}-\frac{1}{q}}=\frac{\frac{1}{2}-\frac{d-2}{2d}}{\frac{d+2}{2d}}=\frac{2}{d+2}.
\end{eqnarray} 
Hence, combining \eqref{s-inequality},  \eqref{interpolation-inequality}, and \eqref{parametr}, we get
\begin{eqnarray*}
\|x\|_{L^{2}(\mathbb{T}^d_\theta)} &\overset{ \text{\eqref{s-inequality} \eqref{interpolation-inequality}} }{\leq}  ( C_{1,d} \|x\|_{H_{2}^{1}(\mathbb{T}^d_\theta)})^{1-\eta}  \|x\|^{\eta}_{L^{1}(\mathbb{T}^d_\theta)} \overset{\text{(\ref{parametr})} }{=} C^{\frac{d}{d+2}}_{1,d} \| x\|^{\frac{d}{d+2}}_{H_{2}^{1}(\mathbb{T}^d_\theta)}\|x\|^{\frac{2}{d+2}}_{L^{1}(\mathbb{T}^d_\theta)}. 
\end{eqnarray*}
 Rising both sides of the last inequality  to the power   $\frac{d+2}{d}$, one gets inequality \eqref{Nash's-inequality}.
\end{proof}

The following result shows an application of the Nash inequality to compute the decay rate for a heat equation. 
\begin{cor}Let $d>2$ and $u_{0}\in L^{1}(\mathbb{T}^d_{\theta})$ be a positive operator such that
\begin{equation}\label{heat-equation}
   \partial_t u(t)=\Delta_{\theta}u (t)-u(t),\quad u(0)=u_0.
\end{equation}
Then the solution of equation \eqref{heat-equation} has the following time decay 
$$
\|u(t)\|_{L^{2}(\mathbb{T}^d_{\theta})}\leq \left(\|u_0\|^{-\frac{4}{d}}_{L^{2}(\mathbb{T}^d_{\theta})}+\frac{4}{dC^{-2}_{1,d}}\|u_0\|^{-\frac{4}{d}}_{L^{1}(\mathbb{T}^d_{\theta})}t\right)^{-\frac{d}{4}},
$$
for all $t\geq 0,$ where $C_{1,d}>0$ is the constant in \eqref{Nash's-inequality}.
\end{cor}
\begin{proof}
If $u_{0}=\sum\limits_{n\in \mathbb{Z}^d}\widehat{u}_0(n)e^{\theta}_{n}$ and $u(t)=\sum\limits_{n\in \mathbb{Z}^d}\widehat{u}(n)e^{\theta}_{n},$
then the solution of equation \eqref{heat-equation} is written as 
\begin{equation}\label{heat-solution}
u(t)=\sum_{n\in \mathbb{Z}^d}e^{-t(|n|^2+1)}\widehat{u}_0(n)e^{\theta}_{n},
\end{equation}
where $\widehat{u}_0(n)=\tau_{\theta}(u_{0}(e^{\theta}_{n})^*),$ $n\in\mathbb{Z}^d.$ Therefore,
\begin{eqnarray*}
\frac{d}{dt}\|u(t)\|^2_{L^{2}(\mathbb{T}^d_{\theta})}&=&\tau_{\theta}((\partial_tu(t)) u(t))+\tau_{\theta}(u(t)(\partial_t u(t))\\
&=&2\tau_{\theta}((\partial_t u(t))u(t))\\
&=&-2\tau_{\theta}((-\Delta_{\theta}u(t)) u(t))-2\tau_{\theta}(|u(t)|^2)\\
&=&-2 \|(-\Delta_{\theta})^\frac{1}{2}u(t)\|^2_{L^{2}(\mathbb{T}^d_{\theta})}-2\|u(t)\|^2_{L^{2}(\mathbb{T}^d_{\theta})}\\
&\asymp&-2\| u(t)\|^2_{H_{2}^{1}(\mathbb{T}^d_\theta)}. 
\end{eqnarray*}
Since $u_0$ is a positive operator by assumption, the solution operator $u(t)$ in \eqref{heat-solution} is also positive (see \cite[p. 765]{CXY}). Hence, 
\begin{eqnarray*}
\tau_{\theta}(u(t)){=}\sum\limits_{n\in\mathbb{Z}^d}e^{-t(1+|n|^2)} \widehat{u}_0(n) \tau_{\theta}(e^{\theta}_{n})=e^{-t}\widehat{u}_0(0)=e^{-t}\tau_{\theta}(u_0),  
\end{eqnarray*}
which implies that 
\begin{equation}\label{solution-ineq}
\|u(t)\|_{L^1(\mathbb{T}^d_{\theta})} = e^{-t}\|u_0\|_{L^1(\mathbb{T}^d_{\theta})}\le\|u_0\|_{L^1(\mathbb{T}^d_{\theta})}, \,\ t>0.
\end{equation}

Set $y(t):= \|u(t)\|^2_{L^{2}(\mathbb{T}^d_{\theta})}.$ Then applying the Nash inequality \eqref{Nash-inequality} with $s=1$, and by  \eqref{solution-ineq} we obtain 
\begin{eqnarray*}
y'\overset{\eqref{Nash-inequality}}{\leq}  -2 C^{-2}_{1,d}\|u(t)\|^{-\frac{4}{d}}_{L^{1}(\mathbb{T}^d_{\theta})}y^{1+\frac{2}{d}}\overset{\eqref{solution-ineq}}{\le}  -2  C^{-2}_{1,d}\|u_0\|^{-\frac{4}{d}}_{L^{1}(\mathbb{T}^d_{\theta})}y^{1+\frac{2}{d}}.
\end{eqnarray*}
Integrating with respect to $t>0,$ we obtain the following estimate 
\begin{eqnarray*}
\|u(t)\|_{L^{2}(\mathbb{T}^d_{\theta})}\leq \left(\|u_0\|^{-\frac{4}{d}}_{L^{2}(\mathbb{T}^d_{\theta})}+\frac{4}{dC^{-2}_{1,d}}\|u_0\|^{-\frac{4}{d}}_{L^{1}(\mathbb{T}^d_{\theta})}t\right)^{-\frac{d}{4}},
\end{eqnarray*}
which completes the proof.   
\end{proof}

\section{Nikolskii inequality}\label{sc3}
In this section, we prove Nikolskii inequality for trigonometric polynomials on the non-commutative torus. Further, we apply this inequality to study embedding properties between Besov, Wiener, and Beurling spaces. 

Let $\Lambda$ be a finite subset of $\mathbb{Z}^d.$ Define a trigonometric polynomial $x_{\Lambda}$ on $L^{\infty}(\mathbb{T}_{\theta}^{d})$ by
\begin{eqnarray}\label{trig-polynom}
x_{\Lambda}:=\sum\limits_{m\in\Lambda}\widehat{x}_{\Lambda}(m)
  e^\theta_m, 
\end{eqnarray}
 where
 $\widehat{x}_{\Lambda}(m)=\tau_{\theta}(x_{\Lambda}(e_{m}^{\theta})^*), \,\ m\in\Lambda.$ The number of elements of $\Lambda\subset \mathbb{Z}^d$ is denoted by $\sharp(\Lambda).$ We also define
 $$\text{supp}(\widehat{x}_{\Lambda}):=\{m\in\mathbb{Z}^d: \widehat{x}_{\Lambda}(m)\neq 0\}.$$ We denote by $T^{\Lambda}$ the set of trigonometric polynomials as in \eqref{trig-polynom}. Then, it is not difficult to see that $T^{\Lambda}\subset L^p(\mathbb{T}^d_{\theta})$ for all $0<p\leq \infty.$ 
 \begin{rem}\label{powers-trig-polinom}
If $x_{\Lambda}$ be a trigonometric polynomial on $L^{\infty}(\mathbb{T}_{\theta}^{d})$ as in \eqref{trig-polynom}, then it is not difficult to see that $x_{\Lambda}^{*}x_{\Lambda}$ and $x_{\Lambda}^{k}$ are also trigonometric polynomials for any power $k\in\mathbb{Z}_{+},$ respectively.
\end{rem}

The following is an analogue of the classical Nikolskii inequality (see \cite[Proposition, p.147]{HH1987})  on  the non-commutative torus. The proof is a modification of the proof in \cite[Theorem 3.1, p. 989]{NRT}. 
\begin{thm}\label{thm3.2} (Nikolskii  inequality). Let $0<p\leq q\leq\infty$ and let $\gamma$ be the smallest even number larger than or
equal to $\frac{p}{2}$ for $2<p<\infty$ and $\gamma=1$ for $0<p\leq2.$ Let $x_\Lambda$ be a trigonometric polynomial as in \eqref{trig-polynom}. Then we have
\begin{eqnarray}\label{additive3.3}
 \|x_\Lambda\|_{L^q(\mathbb{T}^d_\theta)}\leq {C^{\frac1p-\frac1q}_{\gamma,x_{\Lambda}}}\|x_\Lambda\|_{L^p(\mathbb{T}^d_\theta)},   \end{eqnarray}
where $C_{\gamma,x_{\Lambda}}:=\sharp\{m\in\mathbb{Z}^d: \widehat{x^{\gamma}_{\Lambda}}(m)\neq 0\}$
\end{thm}

\begin{proof} Since the case $p=q$ is clear, we consider the case $p<q.$ We spilt the proof into several steps.

{\it Step 1.} We assume that $p=2$ and $q=\infty$. By  (\ref{additive2.6}) we have 
 \begin{eqnarray}\label{additive3.4}
\|x_\Lambda\|^2_{L^2(\mathbb{T}^d_\theta)}&=&\tau_\theta\left(|x_\Lambda|^2\right)=\tau_\theta\left((x_\Lambda)^*x_\Lambda\right)\overset{ \text{(\ref{additive2.6})} }{=}\sum_{m\in\Lambda} |\widehat{x}_{\Lambda}(m)|^2. 
\end{eqnarray}

By the triangle inequality and using the classical Cauchy-Schwarz inequality for sequences and applying formulas \eqref{additive2.5}, (\ref{trig-polynom}), and  (\ref{additive3.4}), we obtain 
 \begin{eqnarray}\label{additive3.5}
\|x_\Lambda\|_{L^\infty(\mathbb{T}^d_\theta)}&\overset{ \text{(\ref{trig-polynom})} }{=}&\|\sum\limits_{m\in\Lambda} \widehat{x}_{\Lambda}(m)e^\theta_m\|_{L^\infty(\mathbb{T}^d_\theta)}\nonumber\\
&\leq&\sum\limits_{m\in\Lambda} \|\widehat{x}_{\Lambda}(m)e^\theta_m\|_{L^\infty(\mathbb{T}^d_\theta)}\nonumber\\&=& \sum\limits_{m\in\Lambda} |\widehat{x}_{\Lambda}(m)| \|e^\theta_m\|_{L^\infty(\mathbb{T}^d_\theta)}\nonumber\\&\overset{ \text{(\ref{additive2.5})} }{=}& \sum\limits_{m\in\Lambda} |\widehat{x}_{\Lambda}(m)|\cdot 1  
  \nonumber\\
&\leq&C^\frac{1}{2}_{x_{\Lambda}}(\sum\limits_{m\in\Lambda}\left|\widehat{x}_{\Lambda}(m)
\right|^2)^\frac{1}{2} \nonumber\\
  &\overset{ \text{(\ref{additive3.4})} }{=}&C^\frac{1}{2}_{x_{\Lambda}}\|x_\Lambda\|_{L^2(\mathbb{T}^d_\theta)},
\end{eqnarray}
where $C_{x_{\Lambda}}:=\sharp\{m\in\mathbb{Z}^d: \widehat{x}_{\Lambda}(m)\neq 0\}.$

{\it Step 2.} We now suppose that $p=2$ and $2<q< \infty$. Then  
by the inverse Hausdorff-Young inequality (see (\ref{additive3.1'})) and  the  H\"{o}lder inequality with respect to $r=\frac{2}{q'}$ and $\frac{1}{r'}=1-\frac{q'}{2}$ and (\ref{additive3.4}), we obtain  
\begin{eqnarray*}
\|x_{\Lambda}\|_{L^q(\mathbb{T}^d_\theta)}&\overset{ \text{(\ref{additive3.1'})}}{\leq}&
\|\widehat{x}_{\Lambda}\|_{\ell^{q'}(\mathbb{Z}^d)}\\
&=&\left(\sum\limits_{m\in\Lambda}\left|\widehat{x}_{\Lambda}(m)\right|^{q'}\right)^\frac{1}{q'}
\\
&\leq&C^{\frac{1}{q'}-\frac{1}{2}}_{x_{\Lambda}}\left(\sum\limits_{m\in\Lambda}\left|\widehat{x}_{\Lambda}(m)\right|^2\right)^\frac{1}{2}\\
 &\overset{ \text{(\ref{additive3.4}})}{=}&C^{\frac{1}{2}-\frac{1}{q}}_{x_{\Lambda}}\|x_{\Lambda}\|_{L^2(\mathbb{T}^d_\theta)}.
\end{eqnarray*}

{\it Step 3.} Let $p>2$ and $2<p<q\leq \infty$. We assume that $\gamma$ is the smallest even integer such that $\gamma\geq \frac{p}{2}$ and $x^\gamma_{\Lambda} \neq0$. By Remark \ref{powers-trig-polinom}, we obtain that $|x_{\Lambda}|^{2}$ and  $|x_{\Lambda}|^{\gamma}:=(x^{*}_{\Lambda}x_{\Lambda})^{\gamma/2}$ are trigonometric polynomials. Then it follows from Step 1 (see \eqref{additive3.5}) and the non-commutative H\"{o}lder inequality (see \cite[Theorem 1]{Sukochev}) with $\frac{1}{2}=\frac{2}{\infty}+\frac{1}{2}$ that
\begin{eqnarray*}
 \||x_{\Lambda}|^\gamma\|_{L^2(\mathbb{T}^d_\theta)} &=&
 \||x_{\Lambda}|^{\gamma-p/2}|x_{\Lambda}|^{p/2}\|_{L^2(\mathbb{T}^d_\theta)}\\
&\leq&
\||x_{\Lambda}|^{\gamma-p/2}\|_{L^\infty(\mathbb{T}^d_\theta)}\||x_{\Lambda}|^{p/2}\|_{L^2(\mathbb{T}^d_\theta)}\\
&=&\| |x_{\Lambda}|\|^{\gamma-p/2}_{L^\infty(\mathbb{T}^d_\theta)}\||x_{\Lambda}| \|^{p/2}_{L^p(\mathbb{T}^d_\theta)}\\
&=&\| |x_{\Lambda}|^\gamma \|_{L^\infty(\mathbb{T}^d_\theta)}\| |x_{\Lambda}| \|^{-p/2}_{L^\infty(\mathbb{T}^d_\theta)}\| x_{\Lambda} \|^{p/2}_{L^p(\mathbb{T}^d_\theta)}\\
&=&\| |x_{\Lambda}|^\gamma \|_{L^\infty(\mathbb{T}^d_\theta)}\| x_{\Lambda} \|^{-p/2}_{L^\infty(\mathbb{T}^d_\theta)}\| x_{\Lambda} \|^{p/2}_{L^p(\mathbb{T}^d_\theta)}\\
&\overset{ \text{(\ref{additive3.5}})}{\leq}&C_{\gamma,|x_{\Lambda}|}^\frac{1}{2}\| |x_{\Lambda}|^\gamma\|_{L^2(\mathbb{T}^d_\theta)}\| x_{\Lambda}\|^{-p/2}_{L^\infty(\mathbb{T}^d_\theta)}\| x_{\Lambda} \|^{p/2}_{L^p(\mathbb{T}^d_\theta)}.
\end{eqnarray*}
Here and below, we use the fact that $\|x^{\kappa}\|_{L^{\infty}(\mathbb{T}^d_{\theta})}=\|x\|^{\kappa}_{L^{\infty}(\mathbb{T}^d_{\theta})}$ for any normal operators $x\in L^{\infty}(\mathbb{T}^d_{\theta})$ and power $\kappa >0,$ which follows from the spectral theorem. 
Thus, we have 
\begin{eqnarray}\label{additive3.6}
\| x_{\Lambda}\|_{L^\infty(\mathbb{T}^d_\theta)}\overset{\text{\eqref{additive3.5}}}{\leq} C_{\gamma,x_{\Lambda}}^\frac{1}{p} \| x_{\Lambda} \|_{L^p(\mathbb{T}^d_\theta)}.
\end{eqnarray}
Next, suppose  that $p<q<\infty$. Then it follows from (\ref{additive3.6}) that   
\begin{eqnarray*} 
\| x_{\Lambda} \|_{L^q(\mathbb{T}^d_\theta)}&=&\| |x_{\Lambda}|^{1-\frac{p}{q}}|x_{\Lambda}|^{\frac{p}{q}}  \|_{L^q(\mathbb{T}^d_\theta)} \nonumber\\
 &{\leq}& \||x_{\Lambda}|^{1-\frac{p}{q}}\|_{L^\infty(\mathbb{T}^d_\theta)}\||x_{\Lambda}|^{\frac{p}{q}}\|_{L^q(\mathbb{T}^d_\theta)}
 \nonumber\\
 &=& \|x_{\Lambda}\|^{1-\frac{p}{q}}_{L^\infty(\mathbb{T}^d_\theta)}\|x_{\Lambda}\|^{\frac{p}{q}}_{L^p(\mathbb{T}^d_\theta)} \nonumber\\&\overset{ \text{(\ref{additive3.6}})}{\leq}&C_{\gamma,x_{\Lambda}}^{\frac{1}{p} -\frac{1}{q} }\|x_{\Lambda}\|_{L^p(\mathbb{T}^d_\theta)}, 
\end{eqnarray*}
from which the assertion follows.

{\it Step 4.} We assume that  $0 <p\leq  2$ and $p < q \leq\infty,$ and assume $x_{\Lambda}\neq 0.$ Since $\gamma=1,$ repeating the argument in the previous step we obtain
\begin{eqnarray}\label{additive3.6'}
\| x_{\Lambda}\|_{L^\infty(\mathbb{T}^d_\theta)}\overset{\text{\eqref{additive3.6}}}{\leq} C_{x_{\Lambda}}^\frac{1}{p} \| x_{\Lambda} \|_{L^p(\mathbb{T}^d_\theta)}.
\end{eqnarray}
Therefore, it follows from (\ref{additive3.6'}) that   
\begin{eqnarray*} 
\| x_{\Lambda} \|_{L^q(\mathbb{T}^d_\theta)}&=&\| |x_{\Lambda}|^{1-\frac{p}{q}}|x_{\Lambda}|^{\frac{p}{q}}  \|_{L^q(\mathbb{T}^d_\theta)} \nonumber\\
 &{\leq}& \||x_{\Lambda}|^{1-\frac{p}{q}}\|_{L^\infty(\mathbb{T}^d_\theta)}\||x_{\Lambda}|^{\frac{p}{q}}\|_{L^q(\mathbb{T}^d_\theta)}
 \nonumber\\
 &=& \|x_{\Lambda}\|^{1-\frac{p}{q}}_{L^\infty(\mathbb{T}^d_\theta)}\|x_{\Lambda}\|^{\frac{p}{q}}_{L^p(\mathbb{T}^d_\theta)} \nonumber\\&\overset{ \text{(\ref{additive3.6'}})}{\leq}&C_{x_{\Lambda}}^{\frac{1}{p} -\frac{1}{q} }\|x_{\Lambda}\|_{L^p(\mathbb{T}^d_\theta)}. 
\end{eqnarray*}
This completes the proof.
\end{proof}

\begin{rem}We are not interested in the best possible constants in \eqref{additive3.3} for all polynomials as $x_\Lambda.$ However, if $\theta=0,$ then we have the sharpness of the constant in \eqref{additive3.3} in the case $p = 2$ and $q = \infty,$ which is attained on the Dirichlet kernel (see, \cite[Theorem 3.1]{NRT}).
\end{rem}

\section{Besov spaces on quantum tori}\label{sc7}
In this section, we study Besov spaces from the different point of view, and analyse embedding properties of Besov spaces on $\mathbb{T}^d_\theta$ by using the Nikolskii inequality. For the classical case of quasi-Banach Besov spaces, we refer the reader to \cite{Peetre} and \cite{HH1987}. In the definition of Besov spaces we use characteristic functions of dyadic corridors or dyadic cubes as in \cite[Section 3.5.3]{HH1987}. Similar investigations were done in \cite{NRT} for compact homogeneous manifolds. To this end we introduce 
\begin{eqnarray} \label{finite-sets}
\Lambda_0&:=&\{n\in\mathbb{Z}^d: |n|<2\},\nonumber\\
\Lambda_k&:=&\{n\in\mathbb{Z}^d: 2^{k-1}< |n|<2^{k+1} \};\quad k=1,2,\dots,
\end{eqnarray}
where $|n|$ is the Euclidean length on $\mathbb{Z}^d.$ Let $T^{\Lambda_k}$ be the set of all trigonometric polynomials as in \eqref{trig-polynom}. Each element  $x_{\Lambda_k}\in T^{\Lambda_k}$ will be denoted briefly by $x_{k},$ that is, 
$$x_{0}=\widehat{x}_0(0),\,\ x_{k}=\sum_{2^{k-1}< |n|<2^{k+1}}\widehat{x}_k(n)e^\theta_n, \,\ n\in\mathbb{Z}^d \setminus \{0\}, \quad k=1,2,\dots.$$
Since $T^{\Lambda_k}\subset L^p(\mathbb{T}^d_{\theta})\subset\mathcal{D}'(\mathbb{T}^d_\theta)$ for $1\leq p\leq \infty,$ it follows that all trigonometric polynomials as in \eqref{trig-polynom} are elements of $\mathcal{D}'(\mathbb{T}^d_\theta).$
The Fourier transform extends to $\mathcal{D}'(\mathbb{T}^d_\theta).$ For $F\in \mathcal{D}'(\mathbb{T}^d_\theta)$ and $n\in \mathbb{Z}^d,$ define Fourier coefficient
$\widehat{F}(n):=\langle F,(e_n^{\theta})^*\rangle.$ Then the Fourier series of $F$ converges to $F$ according to any (reasonable) summation method
$F=\sum\limits_{n\in \mathbb{Z}^d}\widehat{F}(n)e^\theta_n$ (see \cite{XXY}). In order not to confuse the notations, we will denote the elements of $\mathcal{D}'(\mathbb{T}^d_\theta)$ again by $x.$ 
The following definition is a non-commutative analogue of the definition in \cite[Theorem 2, p.164]{HH1987}.
\begin{definition}\label{Besov-def} Let $0< p,q\leq\infty$ and $r\in\mathbb{R}$. Let $\{\Lambda_k\}_{k=0}^{\infty}$ be the sequence of finite subsets of $\mathbb{Z}^d$ defined in \eqref{finite-sets}. Then the associated Besov space on  $\mathbb{T}^d_\theta$ is defined by 
\begin{eqnarray*}
B_{p,q}^r (\mathbb{T}^d_\theta)&=&\{x\in \mathcal{D}'(\mathbb{T}^d_\theta): \exists 
\{x_k\}_{k=0}^{\infty}\subset C^{\infty}(\mathbb{T}^d_{\theta}) \,\ \text{with}\,\ x_k\in T^{\Lambda_k},k\in\mathbb{Z}_+ , \,\ \text{such that}\\
&x=&\sum_{k=0}^{\infty} x_k \,\ \text{in}\,\ \mathcal{D}'(\mathbb{T}^d_\theta) \,\ \text{and}\,\ \|x\|_{B_{p,q}^r(\mathbb{T}^d_\theta)}<\infty\}, 
\end{eqnarray*}
where 
\begin{eqnarray}\label{def_norm_besov1}
\|x\|_{B_{p,q}^r(\mathbb{T}^d_\theta)}:=\inf\|\{x_k\}_{k\geq 0}\|_{\ell^{q}_{r}(L^p)}=\inf\left(\sum\limits_{k=0}^{\infty}2^{krq}\big\|x_{k}\|_p^q\right)^\frac{1}{q},  \quad q<\infty, 
\end{eqnarray}
 and 
\begin{eqnarray}\label{def_norm_besov2}
\|x\|_{B_{p,\infty}^r(\mathbb{T}^d_\theta)}:=\inf\|\{x_k\}_{k\geq 0}\|_{\ell^{\infty}_{r}(L^p)}=\inf\left(\sup\limits_{k\geq0}2^{kr}\big\|x_{k}\|_p\right), \quad q=\infty,   
\end{eqnarray}
where the infimum is taken over all admissible representations of $\{x_k\}_{k=0}^{\infty}$ as above. 
\end{definition}
\begin{rem} Note that our Besov spaces in the previous definition (even for $0<p<1$) are approximation spaces in the sense of Devore and Lorentz~\cite{DL}. If $\theta=0,$ then this kind of definition was studied in \cite{HH1987}. Furthermore, 
$$\inf\left(\sum\limits_{k=0}^{\infty}2^{krq}\big\|x_{k}\|_p^q\right)^\frac{1}{q},$$
where the infimum is taken over all admissible representations as in Definition \ref{Besov-def}, is an equivalent quasi-norm in Besov spaces on ordinary $d$-torus in terms of Littlewood-Paley decomposition (see \cite[Theorem 2, p.162]{HH1987}). Also, in the general case when $\theta\neq 0$ and $1\leq p,q\leq \infty,$ we conjecture that our definition is equivalent to \cite[Definition 3.1]{XXY}. Indeed, let $\varphi$ be a Schwartz function on $\mathbb{R}^d$ such that the condition (3.1) in \cite{XXY} holds. Then the sequence $\{\varphi(2^{-k}\cdot)\}_{k\geq 0}$ is a Littlewood-Paley decomposition of $\mathbb{T}^d$. As in \cite[Section 3.1]{XXY}, the Fourier multiplier on $\mathcal{D}'(\mathbb{T}^d_{\theta})$ with the symbol $\varphi(2^{-k}\cdot)$ is defined by 
$$\widetilde{\varphi}_{k}*x=\sum_{m\in\mathbb{Z}^d}\varphi(2^{-k}m)\widehat{x}(m)e_{m}^{\theta}.$$ This sum is finite for every $k=0,1,2,\dots$ . In particular, we have
$\widetilde{\varphi}_{k}*x\in T^{\Lambda_k}$ whenever $\text{supp}(\varphi)\subset\Lambda_{k},$ where $\Lambda_{k}$ is defined by \eqref{finite-sets}. In this case, we conjecture that the norm in \eqref{def_norm_besov1} is equivalent to that in \cite[Definition 3.1]{XXY} for $1\leq p,q\leq \infty.$
\end{rem}
 
The following result describes the basic properties of Besov spaces in Definition \ref{Besov-def}. Similar to some of the results below can be found from \cite{XXY} (see for example \cite[Proposition 3.3 and Theorem 6.2]{XXY}) in  Besov spaces on the non-commutative torus which are defined in terms of Littlewood-Paley decomposition in the case when $1\leq p,q\leq \infty.$ Our approach allows us to extend the range of indices including  the case  $0<p,q<1.$
\begin{thm}\label{embed_Besov}Let $0<p,q\leq \infty$ and $r\in \mathbb{R}.$ We have
\begin{enumerate}[{\rm(i)}]
    \item $B_{p,q_1}^{r+\varepsilon}(\mathbb{T}^d_\theta)\hookrightarrow B_{p,q_1}^{r}(\mathbb{T}^d_\theta)\hookrightarrow B_{p,q_2}^{r}(\mathbb{T}^d_\theta)\hookrightarrow B_{p,\infty}^{r}(\mathbb{T}^d_\theta),$ $\varepsilon>0,$ $0<p\leq \infty,$ $0<q_1\leq q_2\leq\infty;$
    \item $B_{p,q_1}^{r+\varepsilon}(\mathbb{T}^d_\theta)\hookrightarrow B_{p,q_2}^{r}(\mathbb{T}^d_\theta),$ $\varepsilon>0,$ $0<p\leq \infty,$ $1\leq q_2< q_1\leq\infty;$
     \item $B_{p_1,q}^{r_1}(\mathbb{T}^d_\theta)\hookrightarrow B_{p_2,q}^{r_2}(\mathbb{T}^d_\theta),$ $0<p_1\leq p_2 \leq \infty,$ $0<q<\infty,$ $r_2=r_1-d\Big(\frac{1}{p_1}-\frac{1}{p_2}\Big);$
      \item $B_{p_1,1}^{r}(\mathbb{T}^d_\theta)\hookrightarrow L^{p_2}(\mathbb{T}^d_\theta),$ $0<p_1<p_2 \leq \infty,$ $r=d\Big(\frac{1}{p_1}-\frac{1}{p_2}\Big);$
      \item $B_{p,q}^{r}(\mathbb{T}^d_\theta)\hookrightarrow L^{q}(\mathbb{T}^d_\theta),$ $1<p<q < \infty,$ $r=d\Big(\frac{1}{p}-\frac{1}{q}\Big);$
      \item $B_{p,q}^{r}(\mathbb{T}^d_\theta)\hookrightarrow L^{p}(\mathbb{T}^d_\theta),$ $1\leq p\leq \infty,$ $0<q\leq \infty,$ $r>0;$
      \item $B_{p,q}^{r}(\mathbb{T}^d_\theta)\hookrightarrow L^{1}(\mathbb{T}^d_\theta),$  $0<p<1,$ $0<q\leq\infty,$ $r>d(\frac{1}{p}-1).$
\end{enumerate}
    \end{thm}
\begin{proof}
Let us prove (i). 
For any sequence $\{b(k)\}_{k\geq 0 }$ such that $b(k)\geq 0,$ $k\geq 0$ and for any $\varepsilon>0,$  we have the following simple inequality
$$\sup\limits_{k\geq 0}b(k)\leq \Big(\sum_{k\geq 0}b(k)^{q_2}\Big)^{1/q_2}\leq \Big(\sum_{k\geq 0}b(k)^{q_1}\Big)^{1/q_1} \leq \Big(\sum_{k\geq 0}2^{k\varepsilon q_1}b(k)^{q_1}\Big)^{1/q_1}.$$
Fix $x_k$ and setting $b(k):=2^{kr}\|x_k\|_p$ for $k\geq 0$ and $r\in\mathbb{R},$ and applying above inequalities we obtain
\begin{eqnarray*}
\|x\|_{B_{p,\infty}^r(\mathbb{T}^d_\theta)}&\overset{\text{(\ref{def_norm_besov2})}}{=}&\inf\left(\sup\limits_{k\geq0}2^{kr}\big\|x_k\|_p\right)\\
&\leq& \inf\Big(\sum_{k\geq 0}2^{krq_2}\|x_k\big\|_p^{q_2}\Big)^{1/q_2}\overset{\text{(\ref{def_norm_besov1})}}{=}\|x\|_{B_{p,q_2}^r(\mathbb{T}^d_\theta)}\\
&\leq& \inf\Big(\sum_{k\geq 0}2^{krq_1}\|x_k\|_p^{q_1}\Big)^{1/q_1}\overset{\text{(\ref{def_norm_besov1})}}{=}\|x\|_{B_{p,q_1}^r(\mathbb{T}^d_\theta)}\\
&\leq& \inf\Big(\sum_{k\geq 0}2^{k\varepsilon q_1}2^{krq_1}\|x_k\big\|_p^{q_1}\Big)^{1/q_1}\overset{\text{(\ref{def_norm_besov1})}}{=}\|x\|_{B_{p,q_1}^{r+\varepsilon}(\mathbb{T}^d_\theta)}
\end{eqnarray*}
for $x\in B_{p,q_1}^{r+\varepsilon}(\mathbb{T}^d_\theta),$ which is the desired result.

(ii) Let $1\leq q_2< q_1<\infty$ and $\varepsilon>0.$ Then by the H\"{o}lder inequality with respect to $1=\frac{q_2}{q_1}+\frac{q_1-q_2}{q_1}=\frac{1}{\frac{q_1}{q_2}}+\frac{1}{\frac{q_1}{q_1-q_2}},$  we have
\begin{eqnarray}\label{part-ii}
\Big(\sum_{k\geq 0}2^{kq_2r}\|x_k\big\|_p^{q_2}\Big)^{1/q_2}&=&\Big(\sum_{k\geq 0}\frac{2^{kq_2(r+\varepsilon)}}{2^{kq_2\varepsilon}}\|x_k\big\|_p^{q_2}\Big)^{1/q_2}\nonumber\\
&\leq& \Big(\sum_{k\geq 0}2^{k(r+\varepsilon)q_1}\|x_k\big\|_p^{q_1}\Big)^{1/q_1}\cdot\Big(\sum_{k\geq 0}2^{-\frac{kq_1q_2 \varepsilon}{q_1 -q_2}}\Big)^{\frac{q_1 -q_2}{q_1q_2}}.
\end{eqnarray}
Since the series $\sum\limits_{k\geq 0}2^{-\frac{kq_1q_2 \varepsilon}{q_1 -q_2}}$ converges, we have 
$$\Big(\sum\limits_{k\geq 0}2^{-\frac{kq_1q_2 \varepsilon}{q_1 -q_2}}\Big)^{\frac{q_1 -q_2}{q_1q_2}}\leq C<\infty$$
for some constant $C>0.$ By taking infimum from both side of \eqref{part-ii} over such $\{x_k\}_{k=0}^{\infty},$ we find  
\begin{eqnarray*}
\|x\|_{B_{p,q_2}^r(\mathbb{T}^d_\theta)}&\leq&C\cdot\|x\|_{B_{p,q_1}^{r+\varepsilon}(\mathbb{T}^d_\theta)}
\end{eqnarray*}
for all $x\in B_{p,q_1}^{r+\varepsilon}(\mathbb{T}^d_\theta).$

(iii) Let $0<p_1\leq p_2 \leq \infty,$  $0<q<\infty,$ and $ r_2=r_1-d\Big(\frac{1}{p_1}-\frac{1}{p_2}\Big).$ Then  by the Nikolskii inequality from Theorem \ref{thm3.2} we have 
\begin{eqnarray*}
\sum_{k\geq 0}2^{kqr_2}\|x_k\|_{p_2}^{q}\overset{\text{(\ref{additive3.3})}}{\lesssim}\sum_{k\geq 0}2^{kqr_2}\cdot 2^{kqd(\frac{1}{p_1}-\frac{1}{p_2})}\|x_k\|_{p_1}^{q}
=\sum_{k\geq 0}2^{kq\big(r_2+d(\frac{1}{p_1}-\frac{1}{p_2})\big)}\|x_k\|_{p_1}^{q}. 
\end{eqnarray*}
 Thus,    taking infimum again over such $\{x_k\}_{k=0}^{\infty},$ we obtain
 $$ 
\|x\|^q_{B_{p_2,q}^{r_2}(\mathbb{T}^d_\theta)}\lesssim\|x\|^q_{B_{p_1,q}^{r_1}(\mathbb{T}^d_\theta)}\quad  \forall x \in B_{p_1,q}^{r_1},(\mathbb{T}^d_\theta). 
$$

(iv) By the Nikolskii inequality from Theorem \ref{thm3.2} we obtain
\begin{eqnarray*}\|x\|_{L^{p_2}(\mathbb{T}^d_\theta)}& \leq&\Big\|\sum_{k\geq 0}x_k\Big\|_{p_2}\lesssim\sum_{k\geq 0}\|x_k\|_{p_2}\lesssim\sum_{k\geq 0}2^{kd(\frac{1}{p_1}-\frac{1}{p_2})}\|x_k\|_{p_1} 
\end{eqnarray*}
where $r=d(\frac{1}{p_1}-\frac{1}{p_2}).$ Then, taking infimum   over  $\{x_k\}_{k=0}^{\infty},$ we have  
$$
\|x\|_{L^{p_2}(\mathbb{T}^d_\theta)} \lesssim \|x\|_{B_{p_1,1}^{r}(\mathbb{T}^d_\theta)},\quad \forall x\in B_{p_1,1}^{r}(\mathbb{T}^d_\theta).
$$

(v) follows from (iv). Indeed, let $T$ be such that $T:B_{p,1}^{r}(\mathbb{T}^d_\theta)\hookrightarrow L^{q}(\mathbb{T}^d_\theta).$ Hence, for parameters $p,q,$ and $r$ one can find couples $(q_0, r_0),$ $(q_1,r_1)$ and $\eta\in (0,1)$ satisfying the following conditions
$$d\Big(\frac{1}{p}-\frac{1}{q_0}\Big)=r_0, \quad d\Big(\frac{1}{p}-\frac{1}{q_1}\Big)=r_1, \quad r_0<r<r_1,$$
and 
$$r=(1-\eta)r_0+\eta r_1, \quad \frac{1}{q}=\frac{1-\eta}{q_0}+\frac{\eta}{q_1}.$$
By (iv) we have $T:B_{p,1}^{r_0}(\mathbb{T}^d_\theta)\hookrightarrow L^{q_0}(\mathbb{T}^d_\theta)$ and $T:B_{p,1}^{r_1}(\mathbb{T}^d_\theta)\hookrightarrow L^{q_1}(\mathbb{T}^d_\theta).$ Then repeating the argument in \cite[Theorem 5.6.1, p. 122]{BL1976}, which is also true in our setting, we obtain $$B_{p,q}^{r}(\mathbb{T}^d_\theta)=\Big(B_{p,1}^{r_0}(\mathbb{T}^d_\theta),B_{p,1}^{r_1}(\mathbb{T}^d_\theta)\Big)_{\eta,q}.$$ On the other hand, by \cite[Formula (2.2)]{PXu} we have $$L^{q}(\mathbb{T}^d_\theta)=\Big(L^{q_0}(\mathbb{T}^d_\theta),L^{q_1}(\mathbb{T}^d_\theta)\Big)_{\eta,q}.$$ Thus, it follows that
$$T:B_{p,q}^{r}(\mathbb{T}^d_\theta)=\Big(B_{p,1}^{r_0}(\mathbb{T}^d_\theta),B_{p,1}^{r_1}(\mathbb{T}^d_\theta)\Big)_{\eta,q}\hookrightarrow \Big(L^{q_0}(\mathbb{T}^d_\theta),L^{q_1}(\mathbb{T}^d_\theta)\Big)_{\eta,q}=L^{q}(\mathbb{T}^d_\theta),$$
$r=d(\frac{1}{p}-\frac{1}{q}).$ In other words, we have 
$$B_{p,q}^{r}(\mathbb{T}^d_\theta)\hookrightarrow L^{q}(\mathbb{T}^d_\theta).$$

 (vi) 
Let $r> 0,$ $1\leq p\leq\infty,$ and  $0<q\leq\infty.$ Then by the H\"{o}lder inequality with respect to $1=\frac{1}{q}+\frac{1}{q'}$ and \eqref{def_norm_besov1},  we find
\begin{eqnarray*}
\|x\|_{L^p(\mathbb{T}^d_\theta)}
&\leq&\|\sum\limits_{k\geq0} x_k\big\|_p\lesssim\sum\limits_{k\geq0}\|x_{k} \|_{p}=\sum\limits_{k\geq0} 2^{rk}2^{-rk}\|x_{k}  \|_{p}\\
&\leq&\Big(\sum\limits_{k\geq0}2^{-q'rk}  \Big)^\frac{1}{q'}
\Big(\sum\limits_{k\geq0} 2^{qrk}\|x_{k} \big\|^q_{p}\Big)^\frac{1}{q} \quad\text{for}\quad1\leq q<\infty\\
&\lesssim&  \Big(\sum\limits_{k\geq0}2^{-rk}  \Big) 
 \sup\limits_{k\geq0} 2^{rk}\|x_{k} \big\|_{p}   \quad \text{for} \quad q=\infty \\
&\lesssim& \sum\limits_{k\geq0} \left(2^{qrk} \|x_{k} \|^q_{p}\right)^\frac{1}{q} \leq 
\big(\sum\limits_{k\geq0} 2^{qrk}\|x_{k}  \|^q_{p}\big)^\frac{1}{q} \quad\text{for}  \quad0<q<1.  
\end{eqnarray*}
Here, for the case $0<q<1,$ we use the inequality in \cite[Exercise 1.1.4 (b), p. 11]{G2008}. Then, taking infimum from both sides of above inequalities  over all such $\{x_k\}_{k=0}^{\infty},$ we obtain
\begin{eqnarray*}
\|x\|_{L^p(\mathbb{T}^d_\theta)} 
\lesssim \|x\|_{B_{p,q}^{r}(\mathbb{T}^d_\theta)}, \quad x\in B_{p,q}^{r}(\mathbb{T}^d_\theta). 
\end{eqnarray*}

(vii) Let $r>d(\frac{1}{p}-1),$ $0<p<1,$ and $0<q\leq\infty$. Then applying the Nikolskii inequality in Theorem \ref{thm3.2} with respect to $\Lambda_k$, we have 
\begin{eqnarray*}
\|x\|_{L^1(\mathbb{T}^d_\theta)} 
\lesssim\sum\limits_{k\geq0}\|x_{k} \|_{1}\leq 2^{d(\frac{1}{p}-1)k}\sum\limits_{k\geq0}  \|x_{k} \|_{p} \leq 
\sum\limits_{k\geq0} 2^{rk}\|x_{k}  \|_{p}. 
\end{eqnarray*}
Now, repeating the argument in (vi) we obtain
\begin{eqnarray*}
\inf\left(\sum\limits_{k\geq0} 2^{rk}\|x_{k}  \|_{p}  \right)\overset{\text{(\ref{def_norm_besov1})}}{\lesssim} \|x\|_{B_{p,q}^{r}(\mathbb{T}^d_\theta)}. 
\end{eqnarray*}
Hence, 
\begin{eqnarray*}
\|x\|_{L^1(\mathbb{T}^d_\theta)} 
\lesssim\|x\|_{B_{p,q}^{r}(\mathbb{T}^d_\theta)},  \quad x\in B_{p,q}^{r}(\mathbb{T}^d_\theta),
\end{eqnarray*}
completing the proof.
\end{proof}

\section{Wiener and $\beta$-Wiener spaces on quantum tori}\label{Wiener-spaces}
In this section, we define quantum analogue of Wiener and $\beta$-Wiener spaces. We also study embedding properties between Besov and these spaces. For the classical case when $d=1$ and for details about these spaces we refer to Kahane's book \cite{K-book} (see also  \cite{BLT, Beur, NT3}, \cite[Ch.6]{TB} for $d\geq2$) and for these spaces on homogeneous manifolds we refer the reader to \cite{NRT}. 
\begin{definition} For $k\in\mathbb{Z}_+,$ let $\Lambda_k$ be as in section \ref{sc7}. Define the Wiener space on non-commutative torus as follows:
\begin{eqnarray*}
A(\mathbb{T}^d_\theta)&=&\{x\in \mathcal{D}'(\mathbb{T}^d_\theta):  \exists 
\{x_k\}_{k=0}^{\infty}\subset C^{\infty}(\mathbb{T}^d_{\theta}) \,\ \text{with}\,\ x_k\in T^{\Lambda_k},k\in\mathbb{Z}_+ \,\ \text{such that}\\
&x\overset{\mathcal{D}'}{=}&\sum_{k=0}^{\infty} x_k \,\ \text{and}\,\ \|x\|_{A(\mathbb{T}^d_\theta)}=\inf\left(\sum_{k=0}^{\infty} \sum\limits_{n\in\Lambda_k}|\widehat{x}_k(n)|\right)<\infty\},
\end{eqnarray*}
where the infimum is taken over all admissible representations of such $\{x_k\}_{k=0}^{\infty}.$
\end{definition}
In the classical case when $\theta=0$, for periodic functions such spaces on $\mathbb{T}^d$ have been investigated, for example, in \cite{BLT, Beur, NT3}, and \cite[Ch.6]{TB}. Our definition is approximation variant of Wiener spaces. In particular, when $\theta=0$ the space in the above definition coincides with the classical Wiener space on $\mathbb{T}^d$ with equivalent norms.
\begin{thm}\label{wiener-space} We have
$$\|x\|_{A(\mathbb{T}_{\theta}^d)}\lesssim\|x\|_{B_{2,1}^{d/2}(\mathbb{T}_{\theta}^d)}, \quad x\in B_{2,1}^{d/2}(\mathbb{T}_{\theta}^d).$$
\end{thm}
\begin{proof}Write
$$\|x\|_{A(\mathbb{T}_{\theta}^d)}\leq\sum_{k=0}^{\infty}F_k,$$
where $F_k:=\sum\limits_{n\in\Lambda_k}|\widehat{x}_k(n)|, k\geq 0.$  Hence, by the H\"{o}lder inequality and the Plancherel identity \eqref{additive2.8} we have
\begin{eqnarray*}
F_k\leq \Big(\sum\limits_{n\in\Lambda_k}1\Big)^{1/2} \Big(\sum\limits_{n\in\Lambda_k}|\widehat{x}_k(n)|^2\Big)^{1/2} \overset{\text{(\ref{additive2.8})}}{\lesssim}2^{k\frac{d}{2}}\|x_k\|_{2}.    
\end{eqnarray*}

Therefore,
\begin{eqnarray*}\|x\|_{A(\mathbb{T}_{\theta}^d)}=\inf\left( \sum_{k=0}^{\infty}F_k\right)
\lesssim \inf\left(\sum_{k=0}^{\infty}2^{k\frac{d}{2}}\|x_k\|_{2}\right):=\|x\|_{B_{2,1}^{d/2}(\mathbb{T}_{\theta}^d)},
\end{eqnarray*}
which completes the proof.
\end{proof}
For $0<\beta<\infty,$ let us now study the $\beta$-absolute convergence of the Fourier series on non-commutative torus, which would be defined as in the case $\beta=1$ by
$$
A^{\beta}(\mathbb{T}^d_\theta)=\{x\in \mathcal{D}'(\mathbb{T}^d_\theta):  \|x\|_{A^{\beta}(\mathbb{T}^d_\theta)}=\inf\Big(\sum_{k=0}^{\infty} \sum\limits_{n\in\Lambda_k}|\widehat{x}_k(n)|^{\beta}\Big)^{1/\beta} <\infty\},
$$
where the infimum is taken over all admissible representations of such $\{x_k\}_{k=0}^{\infty}.$
\begin{thm}\label{beta-wiener-space}Let $1<p\leq2.$ Then 
$$\|x\|_{A^{\beta}(\mathbb{T}_{\theta}^d)}\lesssim\|x\|_{B_{p,\beta}^{s d}(\mathbb{T}_{\theta}^d)}, \quad x\in B_{p,\beta}^{\alpha d}(\mathbb{T}_{\theta}^d),$$
for any $s>0$ and $\beta=(s+\frac{1}{p'})^{-1}.$
\end{thm}
\begin{proof}
As previous case, write
$$\|x\|^{\beta}_{A^{\beta}(\mathbb{T}_{\theta}^d)}=\inf\left(\sum_{k=0}^{\infty}F_{k,\beta}\right),$$
where $F_{k,\beta}:=\sum\limits_{n\in\Lambda_k}|\widehat{x}_k(n)|^{\beta}, k\geq0.$  
First, assume that $\beta\equiv(s+\frac{1}{p'})^{-1}\geq 2.$ Applying the Hausdorff-Young  inequality \eqref{additive3.1'}, we obtain
\begin{eqnarray*}
F_{k,\beta}=\sum\limits_{n\in\Lambda_k}|\widehat{x}_k(n)|^{\beta} \overset{\text{(\ref{additive3.1'})}}{\leq} \|x_k\|^{\beta}_{\beta'}.    
\end{eqnarray*}
Applying Nikolskii's inequality from Theorem \ref{thm3.2} for $(L^{\beta'}(\mathbb{T}^d_\theta),L^{p}(\mathbb{T}^d_\theta))$ with respect to $\beta<p',$ or equivalently, $p<\beta',$ we obtain
\begin{eqnarray*}
F_{k,\beta}\leq\|x_k\|^{\beta}_{\beta'}\leq 2^{kd\beta (\frac{1}{p}-\frac{1}{\beta'})}\|x_k\|^{\beta}_{p}=2^{kd\beta (\frac{1}{\beta}-\frac{1}{p'})}\|x_k\|^{\beta}_{p}.    
\end{eqnarray*}
Therefore,
\begin{eqnarray*}
\|x\|^\beta_{A^{\beta}(\mathbb{T}_{\theta}^d)}=\inf\left(\sum_{k=0}^{\infty}F_{k,\beta}\right)
\leq\inf\left(\sum_{k=0}^{\infty}2^{kd\beta (\frac{1}{\beta}-\frac{1}{p'})}\|x_{k}\|^{\beta}_{p}\right)
:=\|x\|^\beta_{B_{p,\beta}^{sd}(\mathbb{T}_{\theta}^d)},
\end{eqnarray*}
where $s=(\frac{1}{\beta}-\frac{1}{p'})$.

Now, let $\beta\equiv(s+\frac{1}{p'})^{-1}<2.$ Set $\kappa:=\frac{p'}{p'-\beta}>1$ and $\delta:=\frac{\beta}{p'}-\frac{\beta}{2}.$ Then we have 
$$\frac{1}{\kappa}+\delta=1-\frac{\beta}{2}=\beta(\frac{1}{\beta}-\frac{1}{2}).$$
Hence, by the H\"{o}lder inequality we have
\begin{eqnarray*}
F_{k,\beta}&=&\sum\limits_{n\in\Lambda_k}|\widehat{x}_k(n)|^{\beta}\leq \Big(\sum\limits_{n\in\Lambda_k}1^{\kappa}\Big)^{1/\kappa} \Big(\sum\limits_{n\in\Lambda_k}|\widehat{x}_k(n)|^{\beta \kappa'}\Big)^{1/\kappa'}\\
&\lesssim& 2^{k\frac{d}{\kappa}}\Big(\sum\limits_{n\in\Lambda_k}|\widehat{x}_k(n)|^{\beta \kappa'}\Big)^{1/\kappa'}=2^{kd\frac{p'-\beta}{p'}}\Big(\sum\limits_{n\in\Lambda_k}|\widehat{x}_k(n)|^{p'}\Big)^{\beta/p'}.    
\end{eqnarray*}
Since $p'\geq 2,$ it follows from the Hausdorff-Young  inequality \eqref{additive3.1'} that
\begin{eqnarray*}
F_{k,\beta}&\lesssim&2^{kd\frac{p'-\beta}{p'}}\big(\sum\limits_{n\in\Lambda_k}|\widehat{x}_k(n)|^{p'}\big)^{\beta/p'} \overset{\text{(\ref{additive3.1'})}}{\lesssim} 2^{kd\frac{p'-\beta}{p'}}\|x_k\|^{\beta}_{p}.    
\end{eqnarray*}
In other words, we have
\begin{eqnarray*}\|x\|^\beta_{A^{\beta}(\mathbb{T}_{\theta}^d)}&=&\inf\left(\sum_{k=0}^{\infty}F_{k,\beta}\right)\leq \inf\left(\sum_{k=1}^{\infty}2^{kd\frac{p'-\beta}{p'}}\|x_k\|^{\beta}_{p}\right)\\
&=&\inf\left(\sum_{k=0}^{\infty}2^{kd\beta (\frac{1}{\beta}-\frac{1}{p'})}\|x_k\|^\beta_{p}\right):=\|x\|^\beta_{B_{p,\beta}^{sd}(\mathbb{T}_{\theta}^d)}  \end{eqnarray*}
for all $x\in B_{p,\beta}^{sd}(\mathbb{T}_{\theta}^d),$ where $s=(\frac{1}{\beta}-\frac{1}{p'})$.
\end{proof}

\begin{rem} 
(i) In the special case, when $\theta=0$ and $d=1,$ Theorem \ref{beta-wiener-space}  gives the result due to Szasz \cite[p. 119]{Peetre}.

(ii) In the case $p=2,$ we obtain the strongest result which follows from the combination of Theorem \ref{beta-wiener-space} and Theorem \ref{embed_Besov} (iii), that is,
$$\|x\|_{A^{\beta}(\mathbb{T}^{d}_{\theta})}\lesssim\|x\|_{B^{s^* d}_{2,\beta}(\mathbb{T}^{d}_{\theta})}\lesssim\|x\|_{B^{s d}_{p,\beta}(\mathbb{T}^{d}_{\theta})},$$
where $s^*, s>0$ and $\beta=(s^*+\frac{1}{2})^{-1}=(s+\frac{1}{p'})^{-1}.$
\end{rem}
The following is the converse result to Theorem \ref{beta-wiener-space}.
\begin{thm}\label{conv-beta-wiener-space} Let $2\leq p<\infty.$ Then
 $$\|x\|_{B_{p,\beta}^{\Tilde{s} d}(\mathbb{T}_{\theta}^d)}\lesssim\|x\|_{A^{\beta}(\mathbb{T}_{\theta}^d)}, \quad x\in A^{\beta}(\mathbb{T}_{\theta}^d),$$
for $\Tilde{s}:=\min\{s,0\}$ and $\beta=(s+\frac{1}{p'})^{-1}>0.$   
\end{thm}
\begin{proof}Let $\beta\leq p'$ (or equivalently $s\geq 0$). In this case $\Tilde{s}=0.$ Since $p\geq 2,$ applying the Hausdorff-Young inequality \eqref{additive3.1'} we obtain 
\begin{eqnarray*}\|x\|^\beta_{B_{p,\beta}^{0}(\mathbb{T}_{\theta}^d)}&=& \inf\left(\sum_{k=0}^{\infty}\|x_k\|^{\beta}_{p} \right)\overset{\text{(\ref{additive3.1'})}} {\leq}\inf\left(\sum_{k=0}^{\infty}\Big(\sum\limits_{n\in\Lambda_k}|\widehat{x}_k(n)|^{p'}\Big)^{\beta/p'}\right) \\
&\leq&\inf\left(\sum_{k=0}^{\infty}\sum\limits_{n\in\Lambda_k}|\widehat{x}_k(n)|^{\beta}\right):=\|x\|^\beta_{A^{\beta}(\mathbb{T}_{\theta}^d)}, \quad x\in A^{\beta}(\mathbb{T}_{\theta}^d), 
\end{eqnarray*}
where in the third line we have used the inequality in \cite[Exercise 1.1.4 (a), p. 11]{G2008} for $0<\frac{\beta}{p'}\leq 1.$
Let now $\beta>p'$ (or equivalently $\beta'<p$). In this case $\Tilde{s}=s<0.$ By using the H\"{o}lder inequality with respect to $\frac{\beta}{\beta-p'},$ we have
\begin{eqnarray*}
\sum\limits_{n\in\Lambda_k}|\widehat{x}_k(n)|^{p'}&\leq& \Big(\sum\limits_{n\in\Lambda_k}1^{\frac{\beta}{\beta-p'}}\Big)^{1-\frac{p'}{\beta}} \Big(\sum\limits_{n\in\Lambda_k}(|\widehat{x}_k(n)|^{p'})^{\frac{\beta}{p'}}\Big)^{p'/\beta}\\
&\lesssim& 2^{kd(1-\frac{p'}{\beta})}\Big(\sum\limits_{n\in\Lambda_k}|\widehat{x}_k(n)|^{\beta }\Big)^{p'/\beta}.    
\end{eqnarray*}
Hence, by this and the Hausdorff-Young inequality \eqref{additive3.1} we obtain
\begin{eqnarray*}\|x\|^\beta_{B_{p,\beta}^{sd}(\mathbb{T}_{\theta}^d)}&=& \inf\left(\sum_{k=0}^{\infty}2^{ksd\beta}\|x_k\|^{\beta}_{p}\right)\\
&\overset{\text{(\ref{additive3.1'})}} {\leq}&\inf\left(\sum_{k=0}^{\infty}2^{ksd\beta}\Big(\sum\limits_{n\in\Lambda_k}|\widehat{x}_k(n)|^{p'}\Big)^{\beta/p'}\right)\\
&\lesssim&\inf\left(\sum_{k=0}^{\infty}2^{ksd\beta}\Big(2^{kd(1-\frac{p'}{\beta})}\Big(\sum\limits_{n\in\Lambda_k}|\widehat{x}_k(n)|^{\beta }\Big)^{p'/\beta}\Big)^{\beta/p'} \right)\\
&=&\inf\left(\sum_{k=0}^{\infty}2^{ksd\beta}2^{kd(\frac{\beta}{p'}-1)}\sum\limits_{n\in\Lambda_k}|\widehat{x}_k(n)|^{\beta } \right)\\
&=&\inf\left(\sum_{k=0}^{\infty}\sum\limits_{n\in\Lambda_k}|\widehat{x}_k(n)|^{\beta} \right)\\
&:=&\|x\|^{\beta}_{A^{\beta}(\mathbb{T}_{\theta}^d)}, \quad x\in A^{\beta}(\mathbb{T}_{\theta}^d), 
\end{eqnarray*}
where we used the fact $s\beta+\frac{\beta}{p'}-1=0$ in the fourth line. 
\end{proof}
\begin{rem} 
(i) If $\beta\geq 2,$ then we have the strongest result when $p=2,$ that is,
$$\|x\|_{B^{sd}_{p,\beta}(\mathbb{T}^{d}_{\theta})}\lesssim\|x\|_{B^{s^* d}_{2,\beta}(\mathbb{T}^{d}_{\theta})}\lesssim\|x\|_{A^{\beta}(\mathbb{T}^{d}_{\theta})},$$
where  $\beta=(s^*+\frac{1}{2})^{-1}=(s+\frac{1}{p'})^{-1}.$ Note that in this case $s^*, s<0.$ In other words, $\Tilde{s}=s$ and $\Tilde{s^*}=s^*.$

(ii) If $\beta\geq 2,$ then we have the strongest result when $p=\beta',$ that is,
$$\|x\|_{B^{sd}_{p,\beta}(\mathbb{T}^{d}_{\theta})}\lesssim\|x\|_{B^{0}_{\beta',\beta}(\mathbb{T}^{d}_{\theta})}\lesssim\|x\|_{A^{\beta}(\mathbb{T}^{d}_{\theta})},$$
where parameters are as in the previous case.

In both cases the result follows from the combination of Theorem \ref{conv-beta-wiener-space} and Theorem \ref{embed_Besov} (iii).
\end{rem}

\section{Beurling and $\beta$-Beurling spaces on quantum tori}
In this section, we define Beurling and $\beta$-Beurling spaces on quantum tori and study their embedding properties. In the classical case ($\theta=0$ and $d=1$), these spaces were introduced by Beurling \cite{Beur} (see also \cite{BLT} for the algebraic properties), and \cite{NRT} for homogeneous manifolds.

\begin{definition} \label{def-Ber_space}Define the Beurling space on non-commutative torus as follows:
\begin{eqnarray*}
A^*(\mathbb{T}^d_\theta)&=&\{x\in \mathcal{D}'(\mathbb{T}^d_\theta):  \exists 
\{x_k\}_{k=0}^{\infty}\subset C^{\infty}(\mathbb{T}^d_{\theta}) \,\ \text{with}\,\ x_k\in T^{\Lambda_k},k\in\mathbb{Z}_+ \,\ \text{such that},\\
&x\overset{\mathcal{D}'}{=}&\sum_{k=0}^{\infty} x_k \,\ \text{and}\,\ \|x\|_{A^*(\mathbb{T}^d_\theta)}=\inf\left(\sum_{k=0}^{\infty} 2^{kd}\sup\limits_{n\in\Lambda_k}|\widehat{x}_k(n)|\right)<\infty\},    \end{eqnarray*} 
where the infimum is taken over all admissible representations of such $\{x_k\}_{k=0}^{\infty}.$
\end{definition}
Similar to the previous section, for any $0<\beta<\infty$ define $\beta$-version of this space as follows
\begin{eqnarray*}
A^{*,\beta}(\mathbb{T}^d_\theta)=\{x\in \mathcal{D}'(\mathbb{T}^d_\theta):  \|x\|_{A^{*,\beta}(\mathbb{T}^d_\theta)}:=\inf \Big(|\widehat{x}_0(0)|^{\beta}+\sum_{k=1}^{\infty}2^{kd}\Big(\sup\limits_{2^{k-1} < |n|}|\widehat{x}_k(n)|\Big)^{\beta}\Big)^{1/\beta}<\infty\}.    
\end{eqnarray*}
In this case, for any $0<\beta<\infty$ we have the following expression of the norm 
\begin{eqnarray}\label{equiv-beurling-norm}
 \|x\|^{\beta}_{A^{*,\beta}(\mathbb{T}^d_\theta)}&:=& \inf\left(|\widehat{x}_{0}(0)|^{\beta}+\sum_{k=1}^{\infty}2^{kd}\Big(\sup\limits_{2^{k-1} < |n|}|\widehat{x}_k(n)|\Big)^{\beta}\right)\nonumber\\
 &\asymp&\inf\left(|\widehat{x}_{0}(0)|^{\beta}+ \sum_{k=1}^{\infty}2^{kd}\Big(\sup\limits_{n\in\Lambda_k}|\widehat{x}_k(n)|\Big)^{\beta}\right).    
\end{eqnarray}
Indeed, the inequality $ ``\geq"$ is trivial. Let us prove the converse inequality. We write
\begin{eqnarray*}
 |\widehat{x}_{0}(0)|^{\beta}+\sum_{k=1}^{\infty}2^{kd}\Big(\sup\limits_{2^{k-1}< |n|}|\widehat{x}_k(n)|\Big)^{\beta}
 &\asymp& |\widehat{x}_{0}(0)|^{\beta}+\sum_{k=1}^{\infty}2^{kd}\sup\limits_{k\leq l}\sup\limits_{n\in\Lambda_l}|\widehat{x}_l(n)|^{\beta}\\
 &\lesssim&|\widehat{x}_{0}(0)|^{\beta}+\sum_{k=1}^{\infty}2^{kd}\sum\limits_{l=k}^{\infty}\sup\limits_{n\in\Lambda_l}|\widehat{x}_l(n)|^{\beta}\\
  &=&|\widehat{x}_{0}(0)|^{\beta}+\sum_{l=1}^{\infty}\sup\limits_{n\in\Lambda_l}|\widehat{x}_l(n)|^{\beta}\sum\limits_{k=1}^{l}2^{kd}\\ &\asymp&|\widehat{x}_{0}(0)|^{\beta}+\sum_{l=1}^{\infty}2^{ld}\sup\limits_{n\in\Lambda_l}|\widehat{x}_{l}(n)|^{\beta},    
\end{eqnarray*}
which is the desired result. So, we can work with either of equivalent expressions in \eqref{equiv-beurling-norm}.
The following result shows embedding properties between Beurling and Besov spaces.
\begin{thm}\label{beurling-space} Let $0<\beta<\infty$ and $p\geq2.$ Then we have
 $$\|x\|_{B_{p,\beta}^{d(\frac{1}{\beta}-\frac{1}{p'})}(\mathbb{T}_{\theta}^d)}\lesssim\|x\|_{A^{*,\beta}(\mathbb{T}_{\theta}^d)}, \quad x\in A^{*,\beta}(\mathbb{T}_{\theta}^d).$$
\end{thm}
\begin{proof}By the Hausdorff-Young inequality \eqref{additive3.1} we have
\begin{eqnarray*}\|x\|_{B_{p,\beta}^{d(\frac{1}{\beta}-\frac{1}{p'})}(\mathbb{T}_{\theta}^d)}&=& \inf\Big(|\widehat{x}_{0}(0)|^{\beta}+\sum_{k=1}^{\infty}2^{kd(1-\frac{\beta}{p'})}\Big\|\sum\limits_{n\in\Lambda_k}\widehat{x}_k(n)e_{n}^{\theta}\Big\|^{\beta}_{p}\Big)^{1/\beta}\\
&\leq&\inf\Big[|\widehat{x}_{0}(0)|^{\beta}+\sum_{k=1}^{\infty}2^{kd(1-\frac{\beta}{p'})}\Big(\sum\limits_{n\in\Lambda_k}|\widehat{x}_k(n)|^{p'}\Big)^{\beta/p'}\Big]^{1/\beta}\\
&\lesssim&\inf\Big[|\widehat{x}_{0}(0)|^{\beta}+\sum_{k=1}^{\infty}2^{kd(1-\frac{\beta}{p'})}\Big(\sum\limits_{n\in\Lambda_k}1\Big)^{\beta/p'}\Big(\sup\limits_{n\in\Lambda_k}|\widehat{x}_k(n)|\Big)^{\beta}\Big]^{1/\beta}\\
&\lesssim&\inf\Big[|\widehat{x}_{0}(0)|^{\beta}+\sum_{k=1}^{\infty}2^{kd}\Big(\sup\limits_{n\in\Lambda_k}|\widehat{x}_k(n)|\Big)^{\beta}\Big]^{1/\beta}\\
&:=&\|x\|_{A^{*,\beta}(\mathbb{T}_{\theta}^d)}, \quad x\in A^{*,\beta}(\mathbb{T}_{\theta}^d), 
\end{eqnarray*}
which concludes the proof.
\end{proof}
In the case $p=2,$ we obtain the following result which follows from Theorems \ref{beta-wiener-space} and \ref{beurling-space}.
\begin{cor}Let $0<\beta<\infty.$ Then
 $$\|x\|_{A^{\beta}(\mathbb{T}_{\theta}^d)}\lesssim\|x\|_{B_{2,\beta}^{d(\frac{1}{\beta}-\frac{1}{2})}(\mathbb{T}_{\theta}^d)}\lesssim\|x\|_{A^{*,\beta}(\mathbb{T}_{\theta}^d)}, \quad x\in A^{*,\beta}(\mathbb{T}_{\theta}^d).$$
 In particular, if $\beta=1,$ then
 $$\|x\|_{A(\mathbb{T}_{\theta}^d)}\lesssim\|x\|_{B_{2,\beta}^{d/2}(\mathbb{T}_{\theta}^d)}\lesssim\|x\|_{A^{*}(\mathbb{T}_{\theta}^d)}, \quad x\in A^{*}(\mathbb{T}_{\theta}^d).$$
\end{cor} 

\section{Interpolation}\label{inter}

Let $X_0$, $X_1$ be a couple of quasi-normed spaces continuously embedded into a topological vector space $\mathcal{A}.$ For any $x\in X_0+X_1$ we define the $K$-functional by
$$
K(x, t; X_0, X_1) := \inf\limits_{
x=z+y}\Big(\|z\|_{X_0} + t\|y\|_{X_1}
\Big), \quad t \geq 0,
$$
where $z\in X_0$ and $y\in X_1.$
For every fixed $t>0$ this defines a quasi-norm on $X_0+X_1$ and $K(x, t; X_0, X_1)$ is a non-negative, increasing and concave function of $t>0,$ which follows from the definition. For more details, see \cite{Holm},\cite{Peetre}, and \cite[Chapter 6]{DL} (see also \cite{BL1976} for normed spaces).
To study intermediate space $X$ for the pair $(X_0, X_1)$, that is, $X_0 \cap X_1 \subset X \subset
X_0 + X_1,$  one uses the $\eta,q$-interpolation spaces.  Let $0 < \eta < 1$ and $0<q\leq \infty.$ We define
$$
(X_0, X_1)_{\eta,q}:=\Big\{ x\in X_0+X_1: \int\limits_0^\infty\big(t^{-\eta}K(x,t)\big)^q\frac{dt}{t}<\infty\Big\},
$$
for $q < \infty;$
$$
(X_0, X_1)_{\eta,\infty}:=\{ x\in X_0+X_1: \|x\|_{(X_0, X_1)_{\eta,\infty}}:=\sup\limits_{0<t<\infty}t^{-\eta}K(x,t) <\infty\},
$$
for $q = \infty.$ Then this space is quasi-normed with 
the quasi-norm
$$\|x\|_{(X_0, X_1)_{\eta,q}}:=\Big(\int\limits_0^\infty\big(t^{-\eta}K(x,t)\big)^q\frac{dt}{t}\Big)^{\frac{1}{q}}.
$$
Moreover, it is an interpolation space (see \cite[Theorem 1.1]{Holm},\cite[Chapter 6]{DL}).

\begin{definition} \label{def-Ber-type-space} We define the following Beurling-type spaces on quantum tori:
\begin{eqnarray*}
A_r^{*,\beta}(\mathbb{T}^d_\theta)&=&\{x\in \mathcal{D}'(\mathbb{T}^d_\theta):  \exists 
\{x_k\}_{k=0}^{\infty}\subset C^{\infty}(\mathbb{T}^d_{\theta}) \,\ \text{with}\,\ x_k\in T^{\Lambda_k},k\in\mathbb{Z}_+ \,\ \text{such that}\\
&x\overset{\mathcal{D}'}{=}&\sum_{k=0}^{\infty} x_k \,\ \text{and}\,\ \|x\|_{A_r^{*,\beta}(\mathbb{T}^d_\theta)}<\infty\},    \end{eqnarray*} 
 where
\begin{eqnarray}\label{def_beru_type}
\|x\|_{A_r^{*,\beta}(\mathbb{T}^d_\theta)}:=\inf\Big(|\widehat{x}_0(0)|^\beta+\sum_{k=1}^{\infty}\big(2^{rkd}\sup\limits_{2^{k-1}< |n| }|\widehat{x}_k(n)|\big)^\beta\Big)^{\frac{1}{\beta}},
\end{eqnarray}
and the infimum is taken over all admissible representations of such $\{x_k\}_{k=0}^{\infty}.$
\end{definition} 
Note that $A_{\frac{1}{\beta}}^{*,\beta}(\mathbb{T}^d_\theta)=A^{*,\beta}(\mathbb{T}^d_\theta),$ where $A^{*,\beta}(\mathbb{T}^d_\theta)$ is defined in Section \ref{Wiener-spaces}. 
\begin{rem}It can be proved that $\beta$-Wiener, $\beta$-Beurling, and the space in Definition \ref{def-Ber-type-space} are quasi-normed spaces. Since we do not use their completeness elsewhere, we do not try to prove that. But, we conjecture that these spaces are indeed quasi-Banach. Although we give the definition of interpolation spaces for general quasi-Banach spaces, in the theorem below we use only their quasi-norms, but not their completeness.   
    \end{rem}
\begin{thm} \label{intepolation_Bru} Let $0 < r_1 < r_0 < \infty,$ $ 0 < \beta_0, \beta_1, q \leq\infty,$  and $r=(1-\eta)r_0+\eta r_1$. 
\begin{enumerate}[{\rm(i)}]
    \item We have 
    $$
    (A_r^{*,\beta_0}(\mathbb{T}^d_\theta), A_r^{*,\beta_1}(\mathbb{T}^d_\theta))_{\eta, q} =A_r^{*,q}(\mathbb{T}^d_\theta).
    $$   
    In particular,   
    $$
    (A^{*,\frac{1}{r_0}}(\mathbb{T}^d_\theta), A^{*,\frac{1}{r_1}}(\mathbb{T}^d_\theta))_{\eta, \frac{1}{r}}=A^{*,\frac{1}{r}}(\mathbb{T}^d_\theta);
    $$    
    \item If $0<p\leq\infty.$ Then
    $$
    (B_{p,\beta_0}^{r_0}(\mathbb{T}^d_\theta), B_{p,\beta_1}^{r_1}(\mathbb{T}^d_\theta))_{\eta, q}=B_{p, q}^{r}(\mathbb{T}^d_\theta).
    $$
\end{enumerate}
\end{thm}

\begin{proof} Let us prove (i). We assume $x \in (A_{r_0}^{*,\beta_0}(\mathbb{T}^d_\theta), A_{r_1}^{*,\beta_1}(\mathbb{T}^d_\theta))_{\eta, q}.$  Take any representation $x = z+y$ such that $z\in A_{r_0}^{*,\beta_0}(\mathbb{T}^d_\theta),$ and $y\in  A_{r_1}^{*,\beta_1}(\mathbb{T}^d_\theta).$  Then for any $k\geq1$,  we get
\begin{eqnarray}\label{11.2}
2^{rkd}\sup\limits_{2^{k-1}<|n|} |\widehat{x}_k(n)|&\leq&
\sup\limits_{2^{k-1}< |n|}\big[(2^{(r-r_0)kd}(2^{r_0kd} |\widehat{z}_k (n)|)   + 2^{(r-r_1)kd}(2^{r_1kd} |\widehat{y}_k(n)|)\big]\nonumber\\
&\lesssim&
2^{(r-r_0)kd}\big(|\widehat{z}_0(0)|^{\beta_0}+\big(2^{r_0kd}\sup\limits_{2^{k-1}<|n|}  |\widehat{z}_k(n)|\big)^{\beta_0}\big)^{\frac{1}{\beta_0}}\nonumber\\
&+&2^{(r-r_1)kd} \big(|\widehat{y}_0(0)|^{\beta_1}+\big(2^{r_1kd}\sup\limits_{2^{k-1}< |n|}|\widehat{y}_k(n)|\big)^{\beta_1}\big)^{\frac{1}{\beta_1}}\nonumber\\
&\leq&
2^{(r-r_0)kd}\Big[\Big(|\widehat{z}_0 (0)|^{\beta_0}+\sum\limits_{s=1}^\infty\big(2^{r_0sd}\sup\limits_{2^{s-1}<|n|}  |\widehat{z}_k(n)|\big)^{\beta_0}\Big)^{\frac{1}{\beta_0}}\nonumber\\
&+&2^{r'kd}\Big(|\widehat{y}_0 (0)|^{\beta_1}+\sum\limits_{s=1}^{\infty}\big(2^{r_1sd}\sup\limits_{2^{s-1}< |n|}|\widehat{y}_k(n)|\big)^{\beta_1}\Big)^{\frac{1}{\beta_1}}\Big]\nonumber\\
&\overset{\text{(\ref{def_beru_type})}}{=}&2^{(r-r_0)kd}\Big(\|z\|_{A_{r_0}^{*,\beta_0}(\mathbb{T}^d_\theta)}+2^{r'kd}\|y\|_{A_{r_1}^{*,\beta_1}(\mathbb{T}^d_\theta)}\Big) 
\end{eqnarray} 
and 
\begin{eqnarray}\label{11.3}
|\widehat{x}_0(0)|&\lesssim&
2^{(r-r_0)d}\big(|\widehat{z}_0(0)|^{\beta_0}+\big(2^{r_0d}\sup\limits_{2^{k-1}< |n|}  |\widehat{z}_k(n)|\big)^{\beta_0}\big)^{\frac{1}{\beta_0}}\nonumber\\
&+&2^{(r-r_1)d} \big(|\widehat{z}_0(0)|^{\beta_1}+\big(2^{r_1d}\sup\limits_{2^{k-1}<|n|}|\widehat{z}_k(n)|\big)^{\beta_1}\big)^{\frac{1}{\beta_1}}\nonumber\\
&\lesssim&2^{(r-r_0)d}\Big(\|z\|_{A_{r_0}^{*,\beta_0}(\mathbb{T}^d_\theta)}+2^{r'd}\|y\|_{A_{r_1}^{*,\beta_1}(\mathbb{T}^d_\theta)}\Big),
\end{eqnarray} 
where $0 < r_1 < r_0 < \infty,$ $ 0 < \beta_0, \beta_1 \leq\infty,$  and $r'=r_0-r_1,$ $r=(1-\eta)r_0+\eta r_1$.

Using \eqref{11.2} and  \eqref{11.3} and the definition $K(x, t) := K(x, t; A_{r_0}^{*,\beta_0}, A_{r_1}^{*,\beta_1}),$  we find 
\begin{eqnarray*}
2^{rdk}\sup\limits_{2^{k-1}<|n|} |\widehat{x}_k(n)|&\overset{\text{(\ref{11.2}) }}{\lesssim}&2^{(r-r_0)kd} K(x, 2^{r'kd})=2^{-\eta r'kd} K(x, 2^{r'kd}) 
\end{eqnarray*}
and 
\begin{eqnarray*}
|\widehat{x}_0(0)|&\overset{\text{(\ref{11.3}) }}{\lesssim}&
2^{(r-r_0)d} K(x, 2^{r'd})=2^{-\eta r'd} K(x, 2^{r'd}), 
\end{eqnarray*}
where $r'=r_1-r_0.$  Therefore, 
\begin{eqnarray*}
\|x\|_{A_{r}^{*,q}(\mathbb{T}^d_\theta)}&\overset{\text{(\ref{def_beru_type})}}{=}&\inf\Big(|\widehat{x}_0(0)|^q+\sum\limits_{k=1}^{\infty}\big(2^{rkd}\sup\limits_{2^{k-1}<|n|} |\widehat{x}_k(n)|\big)^q\Big)^{\frac{1}{q}}\\
&\lesssim&
\Big(\big(2^{-\eta r'd} K(x, 2^{(r_0-r_1)d})\big)^q+\sum\limits_{k=1}^\infty\big(2^{-\eta r'kd}K(x, 2^{r'kd})\big)^q\Big)^{\frac{1}{q}}\\
&\leq&
\Big(\sum\limits_{k=1}^\infty\big(2^{-\eta r'kd}K(x, 2^{r'kd})\big)^q+\sum\limits_{k=1}^\infty\big(2^{-\eta r'kd}K(x, 2^{r'kd})\big)^q\Big)^{\frac{1}{q}}\\
&=&2^{\frac{1}{q}}
\Big(\sum\limits_{k=1}^\infty\big(2^{-\eta r'kd}K(x, 2^{r'kd})\big)^q\Big)^{\frac{1}{q}}\\
&\asymp&\Big(\int\limits_{1}^\infty \big(t^{-\eta r'd}K(x, t^{r'd})\big)^q\frac{dt}{t}\Big)^{\frac{1}{q}}\\
&\leq&\Big(\int\limits_{0}^\infty \big(t^{-\eta r'd}K(x, t^{r'd})\big)^q\frac{dt}{t}\Big)^{\frac{1}{q}}\\
&\overset{\text{(\ref{def_beru_type})}}{=}&\|x\|_{\big(A_{r_0}^{*,\beta_0}, A_{r_1}^{*,\beta_1}\big)_{\eta, q}},
\end{eqnarray*}
which means that 
$
\big(A_{r_0}^{*,\beta_0}(\mathbb{T}^d_\theta), A_{r_1}^{*,\beta_1}(\mathbb{T}^d_\theta)\big)_{\eta, q} \hookrightarrow A_{r}^{*,q}(\mathbb{T}^d_\theta). 
$

Let us show the inverse embedding. Let $x \in  A_{r}^{*,q}(\mathbb{T}^d_\theta),$ $\tau = \min\{\beta_0, \beta_1, q\},$  $r \in\mathbb{Z}^+,$  and $r'=r_1-r_0.$  For $l\in\mathbb{N}$ we denote 
$$
z:=\sum\limits_{|n|\leq2^{l}}|\widehat{x}_k(n)|e_{n}^\theta, \quad  y:=x-z. 
$$

Thus, 
\begin{eqnarray}\label{9.4}
\|z\|_{A_{r_0}^{*,\beta_0}(\mathbb{T}^d_\theta)}&\overset{\text{(\ref{def_beru_type})}}{=}&\inf\Big(|\widehat{z}_0(0)|^{\beta_0}+\sum\limits_{k=1}^\infty\big(2^{r_0kd}\sup\limits_{2^{k-1}< |n|} |\widehat{z}_k(n)|\big)^{\beta_0}\Big)^{\frac{1}{\beta_0}}\nonumber\\
&\leq& \Big(|\widehat{z}_0(0)|^{\tau}+\sum\limits_{k=1}^\infty\big(2^{r_0kd}\sup\limits_{2^{k-1}< |n|} |\widehat{z}_k(n)|\big)^{\tau}\Big)^{\frac{1}{\tau}}\nonumber\\
&\leq& \Big(|\widehat{x}_0(0)|^{\tau}+\sum\limits_{k=1}^{l}\big(2^{r_0kd}\sup\limits_{2^{k-1}< |n|} |\widehat{x}_k(n)|\big)^{\tau}\Big)^{\frac{1}{\tau}}
\end{eqnarray}
 and
\begin{eqnarray}\label{9.5}
\|y\|_{A_{r_1}^{*,\beta_1}(\mathbb{T}^d_\theta)}
&\lesssim& 2^{r_1ld}\sup\limits_{2^{l}\leq|n|}|\widehat{x}_l(n)|+ \Big(|\widehat{y}_0(0)|^{\tau}+\sum\limits_{k=l+1}^{\infty}\big(2^{r_1kd}\sup\limits_{2^{k-1}\leq |n|} |\widehat{x}_k(n)|\big)^{\tau}\Big)^{\frac{1}{\tau}}\nonumber\\
&\lesssim&2^{r'ld}\Big(|\widehat{x}_0(0)|^{\tau}+\sum\limits_{k=0}^{l}\big(2^{r_0kd}\sup\limits_{2^{k-1}<|n|}|\widehat{x}_k(n)|\big)^{\tau}\Big)^{\frac{1}{\tau}}\nonumber\\&+& \Big(|\widehat{x}_0(0)|^{\tau}+\sum\limits_{k=l+1}^{\infty}\big(2^{r_1kd}\sup\limits_{2^{k-1}< |n|}|\widehat{x}_k(n)|\big)^{\tau}\Big)^{\frac{1}{\tau}}.
\end{eqnarray}

Since
$$
K(x, t^{r'd}) = \inf\limits_{x=z+y}\big(\|z\|_{A_{r_0}^{*,\beta_0}(\mathbb{T}^d_\theta)} +t^{r'd}\|y\|_{A_{r_1}^{*,\beta_1}(\mathbb{T}^d_\theta)}
) \leq t^{r'd} \|x\|_{A_{r_1}^{*,\beta_1}(\mathbb{T}^d_\theta)},  
$$
we have 
\begin{eqnarray}\label{9.6}
\|x\|_{\big(A_{r_0}^{*,\beta_0}, A_{r_1}^{*,\beta_1}\big)_{\eta, q}}&=&\Big(\int\limits_0^{\infty}\big(t^\eta K(x,t)\big)^q\frac{dt}{t}\Big)^{\frac{1}{q}}\nonumber\\
&=&\frac{1}{r'}\Big(\int\limits_0^{\infty}\big(t^{-\eta r'd}K(x,t^{r'd})\big)^q\frac{dt}{t}\Big)^{\frac{1}{q}}\nonumber\\
& \lesssim&\Big(\|x\|^q_{A_{r_1}^{*,\beta_1}(\mathbb{T}^d_\theta)}\int\limits_0^{1}\big(t^{-\eta r'd}t^{r'd}\big)^q\frac{dt}{t} + \int\limits_1^{\infty}(t^{-\eta r'd} K(x,t^{r'd}))^q\frac{dt}{t}\Big)^{\frac{1}{q}}\nonumber\\
&\asymp&\|x\|_{A_{r_1}^{*,\beta_1}(\mathbb{T}^d_\theta)}+\Big(\sum\limits_{s=0}^{\infty}\big(2^{-\eta r'sd}K(x,2^{ r'sd})\big)^q\Big)^{\frac{1}{q}}.
\end{eqnarray}
In view of $r_1<r,$ we obtain $\|x\|_{A_{r_1}^{*,\beta_1}(\mathbb{T}^d_\theta)}\lesssim\|x\|_{A_{r}^{*,\beta}(\mathbb{T}^d_\theta)}.$ Then, by 
$$
K(x,2^{r'sd})\leq \|z\|_{A_{r_0}^{*,\beta_0}(\mathbb{T}^d_\theta)}+2^{ r'sd}\|y\|_{A_{r_1}^{*,\beta_1}(\mathbb{T}^d_\theta)}
$$
and  formulas \eqref{9.4}-\eqref{9.6}, we infer 
\begin{eqnarray*}
\|x\|_{(A_{r_0}^{*,\beta_0},A_{r_1}^{*,\beta_1})_{\eta,q}} &\overset{\text{(\ref{9.6})}}{\lesssim}&
\|x\|_{A_{r_1}^{*,\beta_1}(\mathbb{T}^d_\theta)}+\inf\limits_{x=z+y}\Big( \sum\limits_{s=0}^{\infty}\big[2^{-\eta r'sd}\big(\|z\|_{A_{r_0}^{*,\beta_0}(\mathbb{T}^d_\theta)}+2^{ r'sd}\|y\|_{A_{r_1}^{*,\beta_1}(\mathbb{T}^d_\theta)}\big)\big]^q\Big)^{\frac{1}{q}}\\
&\overset{\text{(\ref{9.4})(\ref{9.5})}}{\lesssim}&\|x\|_{A_{r}^{*,\beta}(\mathbb{T}^d_\theta)}
+\inf\limits_{x=z+y}\left( \sum\limits_{k=0}^{\infty}\Big( 2^{-\eta r'qkd}2^{r'ld}\Big[|\widehat{x}_0(0)|^{\tau}\right. \\
&+&\sum\limits_{k=0}^{l}\big(2^{r_0kd}\sup\limits_{2^k<|n|} |\widehat{x}_k(n)|\big)^{\tau}\Big]^{\frac{1}{\tau}}\nonumber\\
&+& \left.\Big[ |\widehat{x}_0(0)|^{\tau}+\sum\limits_{k=l+1}^{\infty}\big(2^{r_1kd}\sup\limits_{2^k\leq |n|} |\widehat{x}_k(n)|\big)^{\tau}\Big]^{\frac{1}{\tau}}\Big)^q\right)^{\frac{1}{q}}.
\end{eqnarray*}
Further, since $\tau\leq q,$ applying Hardy's inequality (see \cite[Theorem 2 (iv)]{B}) we obtain 
$$
\|x\|_{(A_{r_0}^{*,\beta_0},A_{r_1}^{*,\beta_1})_{\eta,q}}\lesssim \|x\|_{A_{r}^{*,\beta}(\mathbb{T}^d_\theta)}.
$$
This completes the proof of (i). 
To show (ii), we first note that $$(B^{r_0}_{p,\beta_0}(\mathbb{T}^d_\theta), B^{r_1}_{p,\beta_1}(\mathbb{T}^d_\theta))_{\eta,q}\hookrightarrow B^{r}_{p,q}(\mathbb{T}^d_\theta)$$ can be proved similarly to the proof of (i) using the embedding 
$$B^{r}_{p,\beta}(\mathbb{T}^d_\theta)\hookrightarrow B^{r}_{p,\infty}(\mathbb{T}^d_\theta)$$ which follows from Theorem \ref{embed_Besov} (ii).
The reverse inequality is also proved similarly to (i) by using the Hardy's inequality \cite[Theorem 2 (iv)]{B}.
 \end{proof}
\begin{rem}Note that the proof of Theorem~\ref{intepolation_Bru}\,(ii) follows from \cite[Theorem~7.5, pp.~197--198]{DL}, 
where the interpolation result was established for the quasi-normed space 
$\ell_{r}^{\,q}(X)$, $r>0$, $0<q\leq\infty$, consisting of all sequences 
$a=(a_n)_{n\ge0}$ with $a_n\in X$, $n=0,1,\dots$, where $X$ is a quasi-normed complete space, 
such that the following quantity is finite:
\begin{equation}\label{7.18}
    \sum_{n=0}^{\infty} \|a_n\|_{X}^{q} < \infty, 
    \qquad 0 < q < \infty,
\end{equation}
(with the usual modification that the sum is replaced by the supremum when $q=\infty$).
\end{rem}

{\bf Conflict of interest.}
We can conceive of no conflict of
interest in the publication of this paper. 

{\bf Data availability.}
No new data were created or analysed during this study. Data sharing is not applicable to this article.

\section{Acknowledgements}
The work was partially supported by the grant No. AP23487088 of the Science Committee of the Ministry of Science and Higher
Education of the Republic of Kazakhstan.
Authors would like to thank to Prof. Fedor Sukochev, Prof. Turdebek Bekjan, and Dr. Dmitriy Zanin for their helpful discussions and comments on non-commutative torus. Authors also would like to thank to Dr. Gihyun Lee for his discussion on the non-commutative torus and for giving us some references including his survey paper \cite{Ponge1}, and to Dr. Vishvesh Kumar for showing us some references on $L^p$-$L^q$ boundedness of Fourier multipliers in various spaces and groups. The authors were partially supported by Odysseus and Methusalem grants (01M01021 (BOF Methusalem) and 3G0H9418 (FWO Odysseus)) from Ghent Analysis and PDE center at Ghent University. The first author was also supported by the EPSRC grants EP/R003025/2 and EP/V005529/1.
Authors thank the anonymous referees for reading the paper and providing thoughtful comments, which improved the exposition of the paper.

\begin{center}

\end{center}

\end{document}